\documentclass{amsart}
\RequirePackage{amsmath}
\RequirePackage{amssymb}
\RequirePackage{amsxtra}
\RequirePackage{amsfonts}
\RequirePackage{latexsym}
\RequirePackage{euscript}
\RequirePackage{amscd}
\RequirePackage{amsthm}
\RequirePackage{xypic}
\xyoption{all}
\usepackage{dsfont}

\newcommand{\Rinfty}{{R_\infty}}
\newcommand{\Rbarinfty}{{\bar{R}_\infty}}

\def\A{\mathbb A}

\def\C{\mathbb C}
\def\F{\mathbb F}

\def\Q{\mathbb{Q}}
\def\R{\mathbb{R}}
\def\T{\mathbb{T}}

\def\Z{\mathbb{Z}}
\def\Fbar{\overline{\F}}

\def\m{\mathfrak m}

\newcommand{\cycle}{Z}

\def\chibar{\overline{\chi}}

\def\pr{\mathrm{pr}}

\def\ab{\mathrm{ab}}

\def\GL{\mathrm{GL}}

\def\Gal{\mathrm{Gal}}

\def\Sym{\mathrm{Sym}}
\def\Ext{\mathrm{Ext}}

\def\End{\mathrm{End}}

\def\Art{\mathop{\mathrm{Art}}\nolimits}

\def\Hom{\mathop{\mathrm{Hom}}\nolimits}

\def\Spec{\mathop{\mathrm{Spec}}\nolimits}

\def\Frob{\mathop{\mathrm{Frob}}\nolimits}
\def\Supp{\mathop{\mathrm{Supp}}\nolimits}

\def\Ind{\mathop{\mathrm{Ind}}\nolimits}

\def\rhobar{\overline{\rho}}

\def\cotimes{\widehat{\otimes}}
\def\cboxtimes{\widehat{\boxtimes}}

\def\St{\mathrm{St}}

\def\m{\mathfrak{m}}

\def\iso{\buildrel \sim \over \longrightarrow}

\def\triv{\mathds{1}}

\swapnumbers

\newcommand{\onto}{\twoheadrightarrow}

\newcommand{\into}{\hookrightarrow}

\newcommand{\To}{\longrightarrow}
\newcommand{\isoto}{\stackrel{\sim}{\To}}

\newlength{\ownl}

\newcommand{\ad}{{\operatorname{ad}\,}}

\newcommand{\diag}{{\operatorname{diag}}}

\newcommand{\rec}{{\operatorname{rec}}}

\newcommand{\Res}{{\operatorname{Res}}}

\newcommand{\wt}[1]{\widetilde{#1}}
\newcommand{\GSp}{\operatorname{GSp}}

\newcommand{\PSL}{\operatorname{PSL}}

\newcommand{\cris}{{\operatorname{cr}}}

\newcommand{\loc}{{\operatorname{loc}}}

\newcommand{\semis}{{\operatorname{ss}}}

\newcommand{\univ}{{\operatorname{univ}}}

\newcommand{\CM}{{\mathcal{M}}}

\newcommand{\cC}{\mathcal{C}}

\newcommand{\cG}{\mathcal{G}}
\newcommand{\cI}{\mathcal{I}}

\newcommand{\caL}{\mathcal{L}}
\newcommand{\cM}{\mathcal{M}}

\newcommand{\cO}{\mathcal{O}}

\newcommand{\cS}{\mathcal{S}}

\newcommand{\cX}{\mathcal{X}}
\newcommand{\cY}{\mathcal{Y}}
\newcommand{\cZ}{\mathcal{Z}}

\newcommand{\tQ}{\widetilde{{Q}}}

\newcommand{\tS}{\widetilde{{S}}}
\newcommand{\tT}{\widetilde{{T}}}

\newcommand{\tv}{{\widetilde{{v}}}}

\newcommand{\varepsilonbar  }{\overline{\varepsilon}}

 \newcommand{\phibar    }{\overline{\phi}}   
    
\newcommand{\psibar   }{\overline{\psi}}

 \newcommand{\taut     }{\widetilde{\tau}}

 \newcommand{\omegat   }{\widetilde{\omega}}

\def\RCS$#1: #2 ${\expandafter\def\csname RCS#1\endcsname{#2}}
\RCS$Revision: 135 $
\RCS$Date: 2013-03-20 03:16:33 +0000 (Wed, 20 Mar 2013) $

 \newcommand{\bigO}{\mathcal{O}}

 \newcommand{\p}{\mathfrak{p}} 

\newcommand{\fp}{\mathfrak{p}}
\newcommand{\fq}{\mathfrak{q}}

\newcommand{\Favoid}{F^{(\mathrm{avoid})}} 

\newcommand{\bb}{\mathbb} 
\newcommand{\mc}{\mathcal}
\newcommand{\mf}{\mathfrak}

\DeclareMathOperator{\ssg}{ss}

\newcommand{\rbar}{{\bar{r}}}
\newcommand{\sbar}{{\bar{s}}}
\newcommand{\Rbar}{\bar{R}}

\newcommand{\HT}{\operatorname{HT}}

 \newcommand{\Qp}{{\Q_p}}

\newcommand{\GQp}{{G_{\Q_p}}}
\newcommand{\WQp}{{W_{\Q_p}}}
\newcommand{\IQp}{{I_{\Q_p}}}
\newcommand{\Zp}{{\Z_p}}
 
\newcommand{\Qpbar}{{\overline{\Q}_p}}
\newcommand{\Qpbartimes}{{\overline{\Q}^\times_p}}

\newcommand{\Fpbar}{{\overline{\F}_p}}

\newcommand{\Fp}{{\F_p}}

\usepackage{hyperref}

\newtheorem{thm}[subsubsection]{Theorem}
\newtheorem{lemma}[subsubsection]{Lemma}
\newtheorem{lem}[subsubsection]{Lemma}
\newtheorem{df}[subsubsection]{Definition}
\newtheorem{defn}[subsubsection]{Definition}

\newtheorem{conj}[subsubsection]{Conjecture}

\newtheorem{prop}[subsubsection]{Proposition}
\newtheorem{remark}[subsubsection]{Remark}

\newtheorem{terminology}[subsubsection]{Terminology}

\newtheorem{alemma}[subsection]{Lemma}
\newtheorem{aconj}[subsection]{Conjecture}

\newtheorem{acor}[subsection]{Corollary}
\newtheorem{aprop}[subsection]{Proposition}
\newtheorem{aremark}[subsection]{Remark}
\def\numequation{\addtocounter{subsubsection}{1}\begin{equation}}
\def\nummultline{\addtocounter{subsubsection}{1}\begin{multline}}
\def\anumequation{\addtocounter{subsection}{1}\begin{equation}}

\newcommand{\ssinc}{\addtocounter{subsubsection}{1}}
\title[The Breuil--M\'ezard Conjecture]
{A geometric perspective on the Breuil--M\'ezard Conjecture}
\author{Matthew Emerton and Toby Gee}
\thanks{The first author was supported in part by NSF grant
  DMS-1003339, and the second author by NSF grant DMS-0841491}
\address{Mathematics Department, Northwestern
University, 2033 Sheridan Rd., Evanston, IL 60208}
\email[Matthew Emerton]{emerton@math.northwestern.edu}
\email[Toby Gee]{gee@math.northwestern.edu}
\begin{document}
\begin{abstract}Let $p>2$ be prime. We state and prove (under mild
  hypotheses on the residual representation) a geometric refinement of
  the Breuil--M\'ezard conjecture for 2-dimensional mod $p$
  representations of the absolute Galois group of $\Qp$. We also state
  a conjectural generalisation to $n$-dimensional representations of
  the absolute Galois group of an arbitrary finite extension of $\Qp$,
  and give a conditional proof of this conjecture, subject to a certain
  $R = \T$-type theorem together with a strong
  version of the weight part of Serre's conjecture for rank~$n$
  unitary groups. We deduce an unconditional result in the case of
  two-dimensional potentially Barsotti--Tate representations.
\end{abstract}
\maketitle
\section{Introduction}\label{sec:introduction}

Our aim in this paper is to revisit the Breuil--M\'ezard conjecture 
\cite{breuil-mezard} from a geometric point of view. 
Let us explain what we mean by this.
First recall that the Breuil--M\'ezard conjecture posits a
formula (in terms of certain representation-theoretic data) for 
the Hilbert--Samuel multiplicity of the characteristic $p$ fibre of 
certain local $\Z_p$-algebras, namely those whose characteristic zero
fibres parameterize two-dimensional potentially semistable 
liftings of some fixed continuous two-dimensional Galois representation
$\rbar: G_{\Q_p} \to \GL_2(\F)$, where $\F$ is a finite field
of characteristic $p$  (the so-called {\em potentially
semistable deformation rings} constructed in \cite{kisindefrings}).
One way in which a local ring can have multiplicity is if its $\Spec$
has more than one component:  if its $\Spec$ is
the union of $n$ irreducible components, each with multiplicities $\mu_i$
($i = 1,\ldots,n$), then the multiplicity of the entire ring
will be $\sum_i \mu_i$.  Our goal is to
both explain and refine the Breuil--M\'ezard
conjecture in these terms,
by identifying the irreducible components of the various
rings involved, in representation-theoretic terms, as well as to
determine their multiplicities.

To be somewhat more precise, after recalling some background material in 
Section~\ref{sec:cycles}, in Section~\ref{sec:GL2 Qp}
we consider the case of two-dimensional representations
of $G_{\Q_p}$, as introduced above. 
In this case the Breuil--M\'ezard
conjecture is a theorem of Kisin \cite{kisinfmc} (under very mild assumptions
on $\rbar$), and we are able to strengthen Kisin's result so as to prove
our geometric refinement of the conjecture. 
(We give a more detailed description
of our results in this case in Subsection~\ref{subsec:GL2 summary}
below.) The possibility of such an extension is strongly suggested by
the recent paper \cite{breuilmezardIII}, and our results may be viewed
as a sharpening of the results of \emph{ibid}.\ (see Remark
\ref{rem: comparison to BMII} below). In Section~\ref{sec: BM conjecture
for GLn} we propose an extension of the Breuil--M\'ezard conjecture,
and of our geometric refinement thereof, to the
case of $n$-dimensional representations of $G_K$, for any finite extension
$K$ of $\Q_p$ and any positive integer~$n$.   Finally,
in Section~\ref{sec: patching for GLn}, we explain how the arguments
of Section~\ref{sec:GL2 Qp} may be extended to the case of $n$-dimensional representations,
so as to prove an equivalence between the Breuil--M\'ezard conjecture (extended to
the $n$-dimensional case) and its geometric refinement, under the
assumption of a suitable $R =\T$-type theorem, together with a
strong form of the weight part of Serre's conjecture
for rank $n$ unitary groups. In the case of two-dimensional
potentially Barsotti--Tate representations, we deduce an unconditional
geometric refinement of the results of~\cite{geekisin}.

In an appendix we establish a technical result that allows us to realize
representations of local Galois groups as restrictions of automorphic representations
of global Galois groups. \subsection{Summary of our results in the case of two-dimensional representations
of $G_{\Q_p}$.}
\label{subsec:GL2 summary}
We now explain in more detail our geometric
refinement of the original Breuil--M\'ezard conjecture.
To this end,
we fix a finite extension $E$ of $\Q_p$,
with ring of integers $\mathcal O$, residue field $\F$, and uniformizer $\pi$.
As above, we also fix a continuous representation $\rbar: G_{\Q_p} \to \GL_2(\F)$,
and we let $R^{\square}(\rbar)$ denote
the universal lifting ring of $\rbar$ over~$\mathcal O.$
If $m$, $n$ are integers with $n\ge 0$ and $\tau$ is an inertial type
defined over $E$, then we may consider the subset of $\Spec
R^{\square}(\rbar)[1/p]$ consisting of those closed points that
correspond to lifts of $\rbar$ to characteristic zero which are
potentially semistable with Hodge--Tate weights $(m,m+n+1)$ and
inertial type $\tau$. (We adopt the convention that the cyclotomic
character has Hodge--Tate weight $1$, though we caution the reader that
this convention does not remain in force for the entire paper; see
Section \ref{subsec:notation} for the precise conventions we will
follow.) In \cite{kisindefrings},
Kisin proves that there is a reduced closed subscheme $\Spec
R^{\square}(m,n,\tau,\rbar)$ of $\Spec R^{\square}(\rbar)$ such that
this subset is precisely the set of closed points of $\Spec
R^{\square}(m,n,\tau,\rbar)[1/p]$.\footnote{In fact, in
  \cite{kisindefrings} Kisin also fixes the determinants of the lifts
  that he considers, but we will suppress this technical point for
  now.} 
The Breuil--M\'ezard conjecture addresses the problem of describing 
the characteristic $p$ fibre of 
$\Spec R^{\square}(m,n,\tau,\rbar)$, i.e.\ the closed subscheme
$\Spec R^{\square}(m,n,\tau,\rbar)/\pi$
of $\Spec R^{\square}(\rbar)/\pi$.
More precisely, the conjecture as originally stated in \cite{breuil-mezard} gives
a conjectural formula for the Hilbert--Samuel multiplicity of this local scheme.
This conjecture was proved (under very mild assumptions on $\rbar$) in \cite{kisinfmc}.
In this paper we will prove a more precise statement, namely we will identify
the underlying cycle of $\Spec R^{\square}(m,n,\tau,\rbar)/\pi$; that is,
we will describe the irreducible components of this scheme, and the 
multiplicity with which each component appears.  
To explain this more carefully,
suppose first that $\cX$ is any Noetherian scheme.
If $\cZ$ is a closed subscheme of $\cX$,
and $\mathfrak p$ is any point of $\cX$, then we may define the (Hilbert--Samuel) multiplicity $e(\cZ,\mathfrak p)$ of $\cZ$ at $\mathfrak p$
to be the Hilbert--Samuel multiplicity of the stalk $\cO_{\cZ,\mathfrak p}$.
Suppose now that
$\cZ$ is equidimensional of dimension~$d$.  If $\mathfrak a$ is
a point of $\cX$ of dimension~$d$ (i.e.\ whose closure $\overline{\{\mathfrak a\}}$
is of dimension~$d$),
then the stalk $\mathcal O_{\cZ,\mathfrak a}$ is either zero (if $\mathfrak a \not\in\cZ$)
or an Artinian local ring (if $\mathfrak a \in \cZ,$ i.e.\ if $\mathfrak a$ is
a generic point of $\cZ$, or, equivalently, if $\overline{\{\mathfrak a\}}$ is an 
irreducible component of $\cZ$), and the multiplicity $e(\cZ,\mathfrak{a})$ is simply 
the length of $\mathcal O_{\cZ,\mathfrak a}$ as a module over itself,
a quantity which can be interpreted geometrically as the multiplicity with which the
component $\overline{\{\mathfrak a\}}$ appears in $\cZ$.
Since $\cZ$ contains only finitely many generic points, 
the formal sum $\cycle(\cZ):= \sum_{\mathfrak a} e(\cZ,\mathfrak{a}) \mathfrak a$
is well-defined as a $d$-dimensional cycle on $\cX$, and we refer to it as 
the cycle associated to $\cZ$. 
If $\mathfrak p$ is any point of $\cX$, then one has the formula
\ssinc
\begin{equation}
\label{eqn:multiplicities from cycles}
e(\cZ,\mathfrak p) = \sum_{\mathfrak a} e(\cZ,\mathfrak{a})
e\bigl(\bigl(\overline{\{\mathfrak a\}}\bigr),\mathfrak{p}\bigr)
\end{equation}
(where again the sum ranges over points $\mathfrak a$ of dimension~$d$),
allowing us to compute the multiplicity of $\cZ$ at any point in terms of 
its associated cycle.
We are interested in the case when
$\cX := \Spec R^{\square}(\rbar)/\pi$ (an $8$-dimensional Noetherian
local scheme) and $\cZ :=
\Spec R^{\square}(m,n,\tau,\rbar)/\pi$ for some $m$, $n$, $\tau$. It is a theorem of \cite{kisindefrings}
that each of these closed subschemes $\cZ$ is equidimensional of dimension~$5$,
and so we may define the associated cycles $\cycle\bigl(\Spec R^{\square}(m,n,\tau,\rbar)/\pi\bigr)$.
Using this construction, we may in particular define a certain cycle on $\Spec R^{\square}(\rbar)/\pi$
attached to each Serre weight of $\rbar$.
\begin{df}
{\em
If $\sigma$ is a Serre weight of $\rbar$, write
$\sigma=\sigma_{m,n} :=\det^m \otimes\Sym^{n}\F^2$
for integers $m$, $n$ such that $0 \leq n \leq p-1$,
and define
$\cC_{m,n} := \cycle\bigl(\Spec R^{\square}(m,n,\triv,\rbar)/\pi\bigr).$
(To avoid ambiguity, one could insist that $m$ is chosen so
that $0 \leq m \leq p~-~2$.  However,
the subscheme $\Spec R^{\square}(m,n,\triv,\rbar)/\pi$ is in fact
independent of the particular choice of $m$ used to describe $\sigma$.)
}
\end{df}
The following proposition describes these cycles quite explicitly. We
will say that $\cC_{m,n}$ consists of a single component if it has a
single irreducible component, which is also reduced, and that it
consists of two components if it consists of two reduced and
irreducible components.\begin{prop}Assume that the Breuil--M\'ezard Conjecture {\em (}i.e.\
  Conjecture {\em \ref{conj: original BM conjecture for GL2 Qp}} below{\em )} holds for $\rbar$.  
\label{prop:components}
\begin{enumerate}
\item If $\rbar$ is irreducible and $\sigma_{m,n}$ is a Serre weight
  of $\rbar$, then $\cC_{m,n}$ consists of a single component, which
  has multiplicity one at the closed point of $\Spec
  R^{\square}(\rbar)/\pi$.
\item If $\rbar$ is reducible and $\sigma_{m,n}$ is a Serre weight of $\rbar$ such that $n <
  p-2$, or $n=p-2$ and $\rbar$ is a non-split extension of distinct
  characters, then $\cC_{m,n}$ consists of a single component, which has
  multiplicity one at the closed point of $\Spec
  R^{\square}(\rbar)/\pi$.  \item If $\rbar$ is reducible and $\sigma_{m,n}$ is a Serre weight of $\rbar$ with $n = p-1$,
  so that $\rbar|_{I_p}\sim
  \begin{pmatrix}
    \omega^{m+1}&*\\0&\omega^m
  \end{pmatrix}$, then if $*$ is peu ramifi\'ee and $\rbar$ itself is
  a twist of an extension of the trivial character by the mod $p$
  cyclotomic character, then $\cC_{m,n}$ is a
  sum of two components, each having multiplicity one at the closed 
point of $\Spec R^{\square}(\rbar)/\pi$. Otherwise $\cC_{m,n}$ is a single component, having multiplicity one at
  the closed point of $\Spec R^{\square}(\rbar)/\pi$.
\item If $\sigma_{m,n}$ is a Serre weight of $\rbar$ with $n = p-2$
  and $\rbar$ is split and $p$-distinguished, then $\cC_{m,n}$ is a
  sum of two components, each having multiplicity one at the closed
  point of $\Spec R^{\square}(\rbar)/\pi$.
\item\label{item: the p-2 case} If $\sigma_{m,n}$ is a Serre weight of
  $\rbar$ with $n = p-2$ and $\rbar$ has scalar semisimplifiction, then $\cC_{m,n}$
  consists of a single component.
\item If $\sigma_{m,n}$ and $\sigma_{m',n'}$ are distinct Serre
  weights of $\rbar$, then $\cC_{m,n}$ and $\cC_{m',n'}$ have disjoint
  support, except if $m\equiv m'\pmod{p-1}$, $n=0$ and $n'=p-1$ {\em (}possibly after
  interchanging $\sigma_{m,n}$ and $\sigma_{m',n'}${\em )}, in which case
  $\cC_{m,n}$ is equal to a component of $\cC_{m',n'}$.
\end{enumerate}
\end{prop}
We remark that by the results of \cite{kisinfmc}, the hypothesis of
the preceeding proposition holds for most $\rbar$.

We make one more definition before stating our main theorem.
\begin{df}
\label{df:locally algebraic types}
{\em
Given integers $a$, $b$ with $b\ge 0$ and an inertial type $\tau$ (assumed to be
defined over $E$), let $\sigma(\tau)$ denote
the representation of $\GL_2(\Z_p)$ over $E$
associated to $\tau$ via Henniart's inertial local
Langlands correspondence, write $\sigma(a,b,\tau) := (\det^a\otimes \Sym^{b} E^2)
\otimes_E \sigma(\tau)$,
and let $\overline{\sigma(a,b,\tau)}^{\ssg}$ denote the semi-simplification
of the reduction mod $\pi$ of (any) $\GL_2(\Z_p)$-invariant $\cO$-lattice in
$\sigma(a,b,\tau)$.  (The representation so obtained is well-defined
independent of the choice of invariant lattice.)
}
\end{df}
We may now state our geometric refinement of the Breuil--M\'ezard conjecture.
\begin{thm}
\label{thm:main theorem}
Suppose that \[\rbar\not\sim
 \begin{pmatrix}
   \omega\chi&*\\0&\chi
 \end{pmatrix}\] for any character $\chi$. Fix integers $m$, $n$ with $n\ge 0$ and an inertial type $\tau$, 
and for each Serre weight $\sigma_{m,n}$ of $\rbar$,
let $a_{m,n}$ denote the multiplicity with which 
$\sigma_{m,n}$ appears as a constituent of $\overline{\sigma(a,b,\tau)}^{\ssg}$.
Then we have the following equality of cycles:
$$\cycle\bigl(\Spec R^{\square}(a,b,\tau,\rbar)/\pi\bigr)
= \sum_{m,n} a_{m,n} \cC_{m,n}.$$
\end{thm}
\begin{remark}
{\em The usual form of the Breuil--M\'ezard conjecture, as stated in \cite{breuil-mezard}
and proved in \cite{kisinfmc}, can be recovered from this result by applying the
formula~(\ref{eqn:multiplicities from cycles}), and using the explicit description
of the cycles $\cC_{m,n}$ provided by
Proposition~\ref{prop:components}. (Note that while 
Proposition~\ref{prop:components} as stated
assumes that the Breuil--M\'ezard conjecture
holds for $\rbar$, our proof of Proposition~\ref{prop:components} will
actually use Theorem~\ref{thm:main theorem} for $\rbar$ as input, and so
this argument is not circular.)
}
\end{remark}
\begin{remark}
  {\em The hypothesis on $\rbar$ in Theorem \ref{thm:main theorem} is
    slightly weaker than that made in the analogous result in
    \cite{kisinfmc}. This is due to our use of potential modularity
    theorems to realise local representations globally, which is more
    flexible than the construction of \cite{kisinfmc} using CM
    forms. We remark that Pa\v{s}k\={u}nas (\cite{vytasBM12}) has reproved Kisin's
    results (and our generalisation of them) by purely local means under a similarly weakened
    hypothesis on $\rbar$.}
\end{remark}
\begin{remark}
  {\em Theorem \ref{thm:main theorem} is proved via a refinement of the global
    argument made in \cite{kisinfmc}, and uses the local arguments of
    \cite{kisinfmc} (using the $p$-adic Langlands correspondence) as
    an input. In particular, it does not give a new proof of the usual
    form of the Breuil--M\'ezard conjecture.}
\end{remark}
\begin{remark}\label{rem: comparison to BMII}
{\em
In the recent sequel \cite{breuilmezardIII} to their paper \cite{breuil-mezard},
Breuil and M\'ezard have constructed,
for generic $\rbar$,
a correspondence between the irreducible components of $\Spec R^{\square}(a,b,\tau,\rbar)/\pi$ 
and the Serre weights.  We show (in Subsection~\ref{subsec:comparison})
that this coincides with the correspondence
$\sigma_{m,n} \mapsto \cC_{m,n}$.
(Note that when $\rbar$ is generic in the
sense of \cite{breuilmezardIII}, each cycle $\cC_{m,n}$ consists of a
single component.)
Thus our results may be reviewed as a refinement of those of \cite{breuilmezardIII}.
We also note that in \cite{breuilmezardIII}, the authors proceed by
refining the local arguments of \cite{kisinfmc}, while (as already noted) 
in this note we proceed by refining the global arguments of {\em ibid.} 
Thus the approaches of \cite{breuilmezardIII} and of the present note may
be regarded as being somewhat complementary to one another.

}
\end{remark}
\begin{remark}
  {\em In the paper \cite{geekisin}, similar techniques to those of
    this paper are used to prove the Breuil--M\'ezard
    conjecture for two-dimensional potentially Barsotti--Tate
    representations of $G_K$. The key additional ingredients which are
    available in that case, but not in general, are the automorphy
    lifting theorems for potentially Barsotti--Tate representations
    proved in \cite{kis04} and \cite{MR2280776}. The arguments of the present paper are in
    large part based on those of \cite{geekisin}, which in turn relies
    on the strategy outlined in \cite{kisinICM}; in particular, our
    implementation of the patching argument for unitary groups is
    simply the natural adaptation of the arguments of \cite{geekisin}
    to higher rank unitary groups. Theorem~\ref{thm: purely local
      statement for unramified regular pot-BT} below gives a geometric
    refinement of some of the results of~\cite{geekisin}.}
\end{remark}
\subsection{Acknowledgments}
The debt that the arguments of this paper owe to the work of Mark
Kisin will be obvious to the reader; in particular, several of our
main arguments are closely based on arguments from \cite{kisinfmc},
\cite{kisinICM} and \cite{geekisin}. We are also grateful to him, as
well as to Christophe Breuil and Kevin Buzzard, for
helpful comments on an earlier draft of this paper. We would like to
thank Florian Herzig for helpful conversations about the
representation theory of $\GL_n(\F_q)$, and Tom Barnet-Lamb, David
Geraghty, Robert Guralnick, and Florian Herzig for helpful conversations about the material in Appendix
\ref{sec:local to global}. We would like to thank the anonymous
referee for a careful reading, and many helpful comments and suggestions.
\subsection{Notation and Conventions}\label{subsec:notation} Throughout this paper,
$p$ will denote an odd prime.

If $K$ is a field, then we let $G_K$ denote its absolute Galois group. 
If $K$ is furthermore a
finite extension of $\bb{Q}_p$ for some $p$, then we write $I_K$ for the
inertia subgroup of $G_K$.
If $F$ is a number
field and $v$ is a finite place of $F$ then we let $\Frob_v$ denote a
geometric Frobenius element of $G_{F_v}$.

We let $\varepsilon$ denote the $p$-adic cyclotomic character, and
let $\varepsilonbar=\omega$ the mod $p$ cyclotomic character. We denote by
$\omegat$ the Teichm\"uller lift of $\omega$.
We let $\omega_2$ denote a choice of a
fundamental character of $\IQp$ of niveau $2$.

If $K$ is a $p$-adic field, if $\rho$ is a continuous de Rham
representation of $G_K$ over $\Qpbar$, and if $\tau:K \into \Qpbar$,
then we will write $\HT_\tau(\rho)$ for the multiset of Hodge--Tate
numbers of $\rho$ with respect to $\tau$.  By definition, if $W$ is a
de Rham representation of $G_K$ over $\Qpbar$ and if $\tau:K \into
\Qpbar$ then the multiset $\HT_\tau(W)$ contains $i$ with multiplicity
$\dim_{\Qpbar} (W \otimes_{\tau,K} \widehat{\overline{K}}(i))^{G_K}
$. Thus for example $\HT_\tau(\varepsilon)=\{ -1\}$. We will use this
convention throughout the paper, except in Section \ref{sec:GL2 Qp},
where we will use the opposite convention that $\varepsilon$ has
Hodge--Tate weight $1$. We apologise for this, but it seems to us to
be the best way to make what we write compatible with the existing
literature.

Let $K$ be a finite extension of $\Qp$, and let $\rec$ denote the local
Langlands correspondence from isomorphism classes of irreducible
smooth representations of $\GL_n(K)$ over $\C$ to isomorphism classes
of $n$-dimensional Frobenius semisimple Weil--Deligne representations
of $W_K$ defined in \cite{ht}. Fix an isomorphism $\imath:\Qpbar\to\C$. We define the
local Langlands correspondence $\rec_p$ over $\Qpbar$ by $\imath \circ
\rec_p = \rec \circ \imath$. This depends only on
$\imath^{-1}(\sqrt{p})$. We let $\Art_K:K^\times\to W_K^{ab}$ be the isomorphism
provided by local class field theory, which we normalise so that
uniformisers correspond to geometric Frobenius elements. 
We will write $\triv$ for the trivial $n$-dimensional
representation of some group, the precise group and choice of $n$
always being clear from the context.

When discussing deformations of Galois representations,
we will use the terms ``framed deformation'' (which originated
in~\cite{kis04}) and ``lifting'' (which originated in~\cite{cht})
interchangeably.

We write all matrix transposes on the left; so ${}^tA$ is the transpose of $A$.

\section{Background on multiplicities and cycles}
\label{sec:cycles}
In this preliminary section we provide details on the
notions of multiplicities and cycles that we outlined in the introduction.
\subsection{Hilbert--Samuel multiplicities}
Recall that
if $A$ is a Noetherian local ring with maximal ideal $\m$ of dimension~$d$,
and $M$ is a finite $A$-module, then there is polynomial
$P_M^A(X)$ of degree at most $d$ (the Hilbert--Samuel polynomial of $M$),
uniquely determined by the requirement 
that for $n\gg 0$, the value $P_M^A(n)$ is equal to the length of
$M/\m^{n+1}M$ as an $A$-module.
\begin{df}
{\em
The Hilbert--Samuel multiplicity
$e(M,A)$ is defined to be $d!$ times the
coefficient of $X^d$ in $P_M^A(X)$. We write $e(A)$ for $e(A,A)$.
}
\end{df}
Note in particular that if $A$ is Artinian, then $e(M,A)$
is simply the length of $M$ as an $A$-module.
\subsection{Cycles}
Let $\cX$ be a Noetherian scheme. 
\begin{df}
{\em
\begin{enumerate}
\item
Let $\cM$ be a coherent
sheaf on $\cX$, and write $\cZ$ to denote the
scheme-theoretic support of $\cM$ (i.e.\ $\cZ$ is the closed subscheme
of $\cX$ cut out by the annihilator ideal $\cI \subset \cO_X$ of $\cM$).
For any point $x \in \cX$,
we write $e(\cM,x)$ to denote the Hilbert--Samuel multiplicity $e(\cM_x,\cO_{\cZ,x}).$
\item
If $\cZ$ is a closed subscheme of $\cX$, then we write 
$e(\cZ,x) := e(\cO_{\cZ},x)$ for all $x \in X$. 
(Note that $\cZ$ coincides with the scheme-theoretic support of $\cO_{\cZ}$,
and so by definition
this is equal to the multiplicity $e(\cO_{\cZ,x})$ of the local ring
$\cO_{\cZ,x}$.)
\end{enumerate}
}
\end{df}
\begin{remark}
\label{rem:multiplicities}
{\em
In some situations the multiplicity $e(\cM,x)$ is particularly simple to describe.
\begin{enumerate}
\item
If $x$ does not lie in the support of $\cM$, i.e.\ if $\cM_x = 0$,
then $e(\cM,x) = 0$.
\item
If $x$ is a generic point of the scheme-theoretic support $\cZ$ of $\cM$
(so that $\cO_{\cZ,x}$ is an Artinian ring), 
then $e(\cM,x)$ is simply the length of $\cM_x$ as an $\cO_{\cZ,x}$-module.
\end{enumerate}
}
\end{remark}
\begin{df}
{\em
\begin{enumerate}
\item
We say that a point $x \in X$ is
of dimension $d$, and write $\dim(x) = d$, if its closure $\overline{\{x\}}$
is of dimension $d$.
\item
A $d$-dimensional cycle on $\cX$ is a formal finite $\mathbb Z$-linear combination
of points of $\cX$ of dimension $d$.
\item We write $\cX\ge 0$ if $\cX$ is in fact a  $\mathbb Z_{\ge 0}$-linear combination
of points of $\cX$ of dimension $d$, and we write $\cX\ge \cY$ if
$\cX-\cY\ge 0$.\end{enumerate}
}
\end{df}
\begin{df}
\label{df:multiplicities for cycles}
{\em
If $\displaystyle Z = \sum_{\dim(x) = d} n_x x$ is a $d$-dimensional cycle on $\cX$,
then for any point $y \in \cX$, we define the multiplicity $e(Z,y)$ via the
formula
$$e(Z,y) := \sum_{\dim(x) = d} n_x e(\overline{\{x\}},y).$$
(Here $\overline{\{x\}}$ denotes the closure of the point $x$.)
}
\end{df}
\begin{df} 
\label{df:cycles}
{\em
\begin{enumerate}
\item
If $d\geq 0$ is a non-negative integer, and $\cM$ is a coherent sheaf on $\cX$ whose
support has dimension $\leq d$, then we define the $d$-dimensional cycle $\cycle_d(\cM)$
associated to $\cM$ as follows:
$$\cycle_d(\cM) := \sum_{\dim(x) = d} e(\cM,x) x,$$
where, as indicated, the sum ranges over all points of $\cX$ of dimension $d$.
(Our assumption on the dimension of the support of $\cM$ ensures that any point of dimension $d$
lying in the support $\cM$ is necessarily a generic point of that support,
and hence that there are only
finitely many such points lying in the support of $\cM$.
Thus all but finitely many terms appearing in the sum defining $\cycle_d(\cM)$
vanish, and so this sum is in fact well-defined. Note also that, if we let
$\cZ$ denote the scheme-theoretic support of $\cM$,
then by Remark~\ref{rem:multiplicities},
the multiplicity $e(\cM,x)$ is simply the length of $\cM_x$ as an $\cO_{\cZ,x}$-module.)
\item
If the support of $\cM$ is finite-dimensional of dimension $d$, then we write
simply $\cycle(\cM) := \cycle_d(\cM)$.
\item
If $\cZ$ is a closed subset of $\cX$, then we write
$\cycle_d(\cZ) := \cycle_d(\cO_{\cZ})$,
and denote this simply by $\cycle(\cZ)$ if $\cZ$ is finite-dimensional of dimension $d$.
\end{enumerate}
}
\end{df}
\begin{remark}
{\em If $\cZ$ is equidimensional of some finite dimension $d$,
then $\cycle(\cZ)$ encodes the irreducible components of $\cZ$,
together with the multiplicity with which each component appears in
$\cZ$. If $\cZ$ has dimension less than $d$,
then $\cycle_d(\cZ)=0$.
}
\end{remark}
\begin{lemma}
\label{lem:cycles in exact sequences}
If $0 \to \cM' \to \cM \to \cM'' \to 0$ is an exact sequence
of coherent sheaves on $\cX$, such that the support of $\cM$ is of dimension~$\leq d$
{\em (}or equivalently, such that the supports of each of $\cM'$ and $\cM''$ are
of dimension $\leq d${\em )}, then
$$\cycle_d(\cM) = \cycle_d(\cM') + \cycle_d(\cM'').$$
\end{lemma}
\begin{proof}
Let $\cZ$ (resp.\ $\cZ'$ and $\cZ''$) denote the support of $\cM$ (resp.\ of $\cM'$
and $\cM''$), so that $\cZ',\cZ''\subseteq\cZ$.
As already noted in the statement of Definition~\ref{df:cycles},
if $x \in \cX$ is of dimension $d$,
then $e(\cM,x)$ (resp.\ $e(\cM',x)$, resp.\ $e(\cM'',x)$)
is simply the length of $\cM_x$ as an $\cO_{\cZ,x}$-module
(resp.\ the length of $\cM'_x$ as an $\cO_{\cZ',x}$-module,
resp.\ the length of $\cM''_x$ as an $\cO_{\cZ'',x}$-module).
Since each of $\cO_{\cZ',x}$ and $\cO_{\cZ'',x}$ is a quotient
of $\cO_{\cZ,x}$, the claimed additivity of cycles follows from
the additivity of lengths in exact sequences.
\end{proof}
The following lemma records formula~(\ref{eqn:multiplicities from
  cycles}), which was stated in
the introduction.
\begin{lemma}
\label{lem:multiplicities from cycles}
If $\cM$ is a coherent sheaf on $\cX$ with support of dimension $d$,
and if $\cycle(\cM)$
is the cycle associated to $\cM$, 
then for any point $y \in X$, we have the formula
$$e(\cM,y) = e\bigl(\cycle(\cM),y\bigr).$$
\end{lemma}
\begin{proof}
This follows immediately from Theorem 14.7 of \cite{MR1011461}.
\end{proof}

We now establish a cycle-theoretic version of \cite[Prop.~1.3.4]{kisinfmc},
whose proof follows the same lines as the proof of that result.

\begin{lemma}
\label{lem:null}
Let $\cX$ be a Noetherian scheme of finite dimension $d$,
and let $f \in \cO_{\cX}(\cX)$ be regular 
{\em (}i.e.\ a non-zero divisor in each stalk of $\cO_{\cX}${\em )}.
If $\cM$ is an $f$-torsion free coherent sheaf on $\cX$ which
is supported in dimension $d-1$, then $\cM/f\cM$ is supported in
dimension $d-2$.
\end{lemma}
\begin{proof}
Since $\cM$ is $f$-torsion free,
no generic point of $\Supp(\cM)$ is contained in $V(f)$.  Thus $\Supp(\cM/f\cM)$ 
is of dimension $\leq d-2$, as claimed.
\end{proof}

\begin{lemma}
\label{lem:integral}
Let $\cX$ be a Noetherian integral scheme of finite dimension $d$,
and let $f \in \cO_{\cX}(\cX)$ be non-zero.  
If $\cM$ is an $f$-torsion free coherent sheaf on $\cX$ which is generically
free of rank one,
then $\cycle_{d-1}(\cM/f\cM) = \cycle\bigl(V(f)\bigr).$
\end{lemma}
\begin{proof}
Let $x$ be the generic point of $\cX$, and $i_x:\Spec \kappa(x) \to
\cX$ the canonical map.  By assumption $\cM_x$ is one-dimensional
over $\kappa(x)$, and so we may find a morphism of quasi-coherent sheaves
$\cM \to (i_x)_*\kappa(x)$ whose kernel $\cM'$ is torsion.
The image $\cM''$ of this morphism is a coherent subsheaf of $(i_x)_* \kappa(x)$,
and hence is contained in an invertible sheaf $\caL$.
Let $\caL'$ denote the cokernel of the inclusion $\cM'' \hookrightarrow \caL$.

Consider first the exact sequence
$$0 \to \cM' \to \cM \to \cM'' \to 0.$$ All the terms in this sequence
are $f$-torsion free, and so 
$$0 \to \cM'/f\cM' \to \cM/f\cM \to \cM''/f\cM'' \to 0$$ is again exact. 
Lemma~\ref{lem:null} shows that $\cycle_{d-1}(\cM'/f\cM') = 0$, 
and Lemma~\ref{lem:cycles in exact sequences} then implies that
\ssinc
\begin{equation}
\label{eqn:equality}
\cycle_{d-1}(\cM/f\cM) = \cycle_{d-1}(\cM''/f\cM'').
\end{equation}

Next consider the exact sequence $0 \to \caL'[f] \to \caL' \to \caL' \to \caL'/f\caL' \to 0.$
Since all the terms in this exact sequence are torsion, and so supported in dimension~$d-1$,
we see from Lemma~\ref{lem:cycles in exact sequences}
that
\ssinc
\begin{equation}
\label{eqn:equality two}
\cycle_{d-1}(\caL'[f]) = \cycle_{d-1}(\caL'/f\caL').
\end{equation}

Finally, consider the exact sequence
$0 \to \cM'' \to \caL \to \caL' \to 0.$  The first two terms in this sequence are $f$-torsion
free, and so it induces an exact sequence
$0 \to \caL'[f] \to \cM''/f\cM'' \to \caL/f\caL \to \caL'/f\caL' \to 0.$
Again applying Lemma~\ref{lem:cycles in exact sequences}, together with~(\ref{eqn:equality two}),
we find that $\cycle_{d-1}(\cM''/f\cM'') = \cycle_{d-1}(\caL/f\caL)$.
Combining this with~(\ref{eqn:equality}), together with the fact that
$\caL$ is an invertible sheaf,
we find that $\cycle_{d-1}(\cM/f\cM) = \cycle\bigl(V(f)\bigr),$
as required.
\end{proof}

\begin{prop}
\label{prop:cutting out}
Let $\cX$ be a Noetherian scheme of finite dimension $d$,
and $f \in \cO_{\cX}(\cX)$ be regular
{\em (}i.e.\ a non-zero divisor 
in each stalk of $\cO_{\cX}${\em )}.
If $\cM$ is an $f$-torsion free coherent sheaf on $\cX$,
and if $\cycle_d(\CM) = \sum_{\dim(x) = d} n_x x,$
where, as indicated, $x$ runs over the $d$-dimensional points of $\cX$,
then the support of $\cM/f\cM$ has dimension $\leq d-1$,
and
$$\cycle_{d-1}(\cM/f\cM) = \sum_{\dim(x)= d}
n_x \cycle_{d-1}\bigl(\overline{\{x\}} \cap  V(f)\bigr).$$
\end{prop}
\begin{proof}
We argue by induction on the
quantity $n:= \sum_{\dim{x} = d} n_x$,
with the case when $n = 0$
being handled by Lemma~\ref{lem:null}.
Suppose now that $n$ is an arbitrary positive integer,
and choose $x$ of dimension $d$ such that $n_x > 0$.
We may then find a non-zero surjection $\cM_x \to \kappa(x)$ (since
by definition $n_x$ is the length of $\cM_x$), which induces a non-zero
map of quasi-coherent sheaves $\cM \to (i_x)_* \kappa(x)$ (where $i_x$
is the canonical map $\Spec \kappa(x) \to \cX$).
Let $\cM'$ denote the kernel of this map, and $\cM''$ its image, so
that the sequence $0 \to \cM' \to \cM \to \cM'' \to 0$ is exact.

The sheaf $\cM'$ is $f$-torsion free (since it is a subsheaf of $\cM$) 
while the sheaf $\cM''$ is $f$-torsion free, supported on $\overline{\{x\}}$,
and generically free of rank one over this component (being a non-zero
coherent subsheaf of $(i_x)_*\kappa(x)$).
Thus we obtain a short exact sequence
$$0 \to \cM'/f\cM' \to \cM/f\cM \to \cM''/f\cM'' \to 0,$$
and Lemmas~\ref{lem:cycles in exact sequences} and~\ref{lem:integral}
show that
\begin{multline*}
\cycle_{d-1}(\cM/f\cM) = \cycle_{d-1}(\cM'/f\cM') + \cycle_{d-1}(\cM''/f\cM'') \\
= \cycle_{d-1}(\cM'/f\cM') + \cycle_{d-1}\bigl(\overline{\{x\}}\cap V(f)\bigr).
\end{multline*}
The proposition follows by induction.
\end{proof}

We close this section with a result about the product of cycles.  In fact,
we will need to apply such a result in the context of a completed tensor product
of complete Noetherian local $k$-algebras, for some field $k$ (which is fixed for the
remainder of this discussion),
and so we restrict our attention to that particular context.

Note that if $A$ and $B$ are complete Noetherian local $k$-algebras,
and if $\mathfrak p$ and $\mathfrak q$ are primes of $A$ and $B$ respectively,
such that $A/\mathfrak p$ is of dimension $d$ and $B/\mathfrak q$ is of dimension $e$,
then $A/\mathfrak p\cotimes_k B/\mathfrak q$ is a quotient of $A\cotimes_k B$ of
dimension $d + e$.  Hence $\Spec A/\mathfrak p \cotimes_k B/\mathfrak q$ is
a closed subscheme of $\Spec A\cotimes_k B$ of dimension $d + e$, 
and we write 
$$\cycle(\Spec A/\mathfrak p)\times_k \cycle(\Spec B/\mathfrak q) :=  
\cycle(\Spec A/\mathfrak p\cotimes_k B/\mathfrak q),
$$
and then extend this by linearity to a bilinear product from $d$-dimensional cycles
on $\Spec A$ and $e$-dimensional cycles on $\Spec B$ to $(d + e)$-dimensional cycles on
$\Spec A\cotimes_k B$.

If $M$ and $N$ are finitely generated $A$- and $B$-modules respectively,
giving rise to coherent sheaves $\mathcal M$ and $\mathcal N$
on $\Spec A$ and $\Spec B$ respectively, then the completed tensor product
$M\cotimes_k N$ gives rise to a coherent sheaf on $\Spec A\cotimes_k B$,
which we denote by $\mathcal M\cboxtimes \mathcal N$.

\begin{lemma}
\label{lem:product cycles}
In the context of the preceding discussion,
if the support of $\mathcal M$ and $\mathcal N$ are of dimensions 
$d$ and $e$ respectively, then the support of $\mathcal M\cboxtimes \mathcal N$
is of dimension $d + e$,
and
$$
\cycle_{d+e}(\mathcal M \cboxtimes \mathcal N) =
\cycle_d(\mathcal M)\times_k \cycle_e(\mathcal N).
$$
\end{lemma}
\begin{proof}
  This is standard; we sketch the proof. Restricting to the support of
  $\mathcal M\cboxtimes \mathcal N$, we may assume that $M$ (resp.\
  $N$) is a faithful $A$-module (resp.\ $B$-module), so that $A$ has
  dimension $d$ and $B$ has dimension $e$. As in (for example) the
  proof of~\cite[Lemma 3.3(5)]{blght}, the distinct minimal primes of
  $A\cotimes_k B$ are precisely the $\fp(A\cotimes_k
  B)+\fq(A\cotimes_k B)$, where $\fp$ is a minimal prime of~$A$ and
  $\fq$ is a minimal prime of~$B$. We are thus reduced to checking
  that for each such $\fp,\fq$, we have \[e(M \cotimes_k
  N,(A/\fp)\cotimes_k(B/\fq))=e(M,(A/\fp))e(N,(B/\fq)),\]which follows
  from Lech's lemma \cite[Thm.\ 14.12]{MR1011461} exactly as in the
  proof of Proposition~1.3.8 of~\cite{kisinfmc}.
\end{proof}

\section{The geometric Breuil--M\'ezard Conjecture for two-dimensional
  representations of $\GQp$}\label{sec:GL2 Qp}In this section we
explain the Breuil--M\'ezard conjecture in its original setting, that
of two-dimensional representations of $\GQp$, and state our geometric
version. We then recall some of the details of the proof of the
original formulation of the conjecture from \cite{kisinfmc}, and show
that the proof may be extended to prove the geometric version. The one
difference from Kisin's notation and that of the present paper is that
we prefer not to fix one of the Hodge--Tate weights of our Galois
representations to be $0$; this makes no essential difference to any
of the arguments, but the additional flexibility that this notation
gives us is convenient in the exposition. We remind the reader that in
this section, our conventions for Hodge--Tate weights are that
$\varepsilon$ has Hodge--Tate weight $1$.
\subsection{The conjecture}\label{subsec: conjecture for GL2 Qp}
We begin by recalling some notation from \cite{kisinfmc}. Fix a prime
$p>2$ and a finite extension $E$ of $\Q_p$ (our coefficient field),
with ring of integers $\cO$,
residue field $\F$, and uniformiser $\pi$. 
We assume that $\# \F>5$, so that  $\PSL_2(\F)$ is a simple group.

Let
$\rbar:\GQp\to\GL_2(\F)$
be a continuous representation,
and let
$\tau:\IQp\to\GL_2(E)$ be an inertial type, i.e. a representation with
open kernel which extends to $\WQp$. Fix integers $a$, $b$ with $b\ge
0$ and a de Rham character $\psi:\GQp\to\cO^\times$ such that
$\overline{\psi\varepsilon}=\det\rbar$. We let $R^{\square,\psi}(a,b,\tau,\rbar)$ and
$R^{\square,\psi}_\cris(a,b,\tau,\rbar)$ be the framed
deformation
$\cO$-algebras which are universal for framed deformations of $\rbar$
which have determinant $\psi\varepsilon$, and are potentially semistable
(respectively potentially crystalline) with Hodge--Tate weights
$(a,a+b+1)$ and inertial type $\tau$.
As in Section~1.1.2 of \cite{kisinfmc}, we let $\sigma(\tau)$ and $\sigma_\cris(\tau)$ denote
the finite-dimensional irreducible $E$-representations of $\GL_2(\Zp)$
corresponding to $\tau$ via Henniart's inertial local Langlands
correspondence, we set
$\sigma(a,b,\tau)=(\det^a\otimes \Sym^{b}E^2)\otimes_E \sigma(\tau)$ and
$\sigma^\cris(a,b,\tau)=(\det^a\otimes \Sym^{b}E^2)\otimes_E \sigma^{\cris}(\tau)$,
and we let $L_{a,b,\tau}$ (respectively $L_{a,b,\tau}^\cris$) be a $\GL_2(\Zp)$-stable $\cO$-lattice in
$\sigma(a,b,\tau)$ (respectively $\sigma^\cris(a,b,\tau)$). Write $\sigma_{m,n}$ for the representation
$\det^m\otimes \Sym^n\F^2$ of $\GL_2(\Fp)$, $0\le m\le p-2$, $0\le n\le
p-1$, so that we may
write \[(L_{a,b,\tau}\otimes_\cO\F)^\semis\isoto\oplus_{m,n}\sigma_{m,n}^{a_{m,n}},\]and
 \[(L_{a,b,\tau}^\cris\otimes_\cO\F)^\semis\isoto\oplus_{m,n}\sigma_{m,n}^{a^\cris_{m,n}},\]
 for
some integers $a_{m,n}$, $a^\cris_{m,n}$.

We can now state the Breuil--M\'ezard
conjecture \cite{breuil-mezard}.
\begin{conj}\label{conj: original BM conjecture for GL2 Qp}There are integers $\mu_{m,n}(\rbar)$ depending only on
  $m$, $n$, and $\rbar$,
  such that for any $a$, $b$, $\tau$ with $\det\tau=\varepsilon^{1-2a-b}\psi|_{I_\Qp}$, \[e(R^{\square,\psi}(a,b,\tau,\rbar)/\pi)=\sum_{m,n}a_{m,n}\mu_{m,n}(\rbar),\]and \[e(R_\cris^{\square,\psi}(a,b,\tau,\rbar)/\pi)=\sum_{m,n}a^\cris_{m,n}\mu_{m,n}(\rbar).\] 
\end{conj}
\begin{remark}\label{rem: mu is determined by low weight crystalline} {\em By a straightforward twisting argument, the truth
    of the conjecture is independent of the choice of $\psi$, so we
    may assume that $\psi$ is crystalline. Note that if the conjecture
    is true for all $a$, $b$, $\tau$, then it is easy to see that
    $\mu_{m,n}(\rbar)=0$ unless
    $\det\rbar|_{I_{\Qp}}=\omega^{2m+n+1}$, and that if this holds
    then we must have
    $\mu_{m,n}(\rbar)=e(R_\cris^{\square,\psi}(\widetilde{m},n,\triv,\rbar)/\pi)$,
    where $\widetilde{m}$ is chosen so that
    $\psi|_{I_{\Qp}}=\varepsilon^{2\widetilde{m}+n+1}$.  So all the values
    $\mu_{m,n}(\rbar)$ are determined by the crystalline deformation
    rings in low weight.  }
\begin{remark}
  {\em Conjecture \ref{conj: original BM conjecture for GL2 Qp} was
    proved in \cite{kisinfmc} under the additional assumptions that $\rbar\not\sim
    \begin{pmatrix}
      \omega\chi&*\\0&\chi
    \end{pmatrix}$ for any $\chi$, and that if $\rbar$ has scalar
    semisimplification then $\rbar$ is scalar.}
\end{remark}
\end{remark}
We may now state our geometric version of the Breuil--M\'ezard conjecture.
\begin{conj}
  \label{conj: refined BM conjecture for GL2 Qp}For each $0\le m\le
  p-2$, $0\le n\le p-1$, there is a cycle $\cC_{m,n}$ depending only
  on $m$, $n$, and $\rbar$, such
  that for any $a$, $b$, $\tau$ with $\det\tau=\varepsilon^{-2a-b-1}\psi|_{I_\Qp}$, \[\cycle(R^{\square,\psi}(a,b,\tau,\rbar)/\pi)=\sum_{m,n}a_{m,n}\cC_{m,n},\]and \[\cycle(R_\cris^{\square,\psi}(a,b,\tau,\rbar)/\pi)=\sum_{m,n}a^\cris_{m,n}\cC_{m,n}.\] 
\end{conj}
\begin{remark}
{\em
  Note again that if the conjecture is true for all choices of $a$, $b$, $\tau$
  then with the assumptions and notation of Remark~\ref{rem: mu is determined by low
    weight crystalline} we must
  have $\cC_{m,n}=\cycle(R_\cris^{\square,\psi}(\widetilde{m},n,\triv,\rbar)/\pi)$.
}
\end{remark}The following is our main result towards Conjecture
\ref{conj: refined BM conjecture for GL2 Qp}.
\begin{thm}\label{thm: the refined GL2 conjecture holds} If $\rbar\not\sim
    \begin{pmatrix}
      \omega\chi&*\\0&\chi
    \end{pmatrix}$ for any $\chi$, then Conjecture~{\em \ref{conj: refined
      BM conjecture for GL2 Qp}} holds for $\rbar$.
  \end{thm}
\subsection{The proof of Theorem~\ref{thm: the refined GL2 conjecture holds} via patching}
\label{subsec: patching for GL2 Qp}
In this subsection we present the proof of Theorem~\ref{thm: the refined GL2 conjecture holds}.
To this end,
we fix a continuous representation $\rbar:G_\Qp\to\GL_2(\F)$. We begin by realising this representation as the restriction of a
global Galois representation, by a similar (but simpler) argument to that of
Appendix A of \cite{geekisin}.
\begin{prop}
  \label{prop:mod p local representation realised globally}There is a
  totally real field $F$ and a continuous irreducible representation
  $\rhobar:G_F\to\GL_2(\F)$ such that
  \begin{enumerate}
  \item $p$ splits completely in $F$;
  \item $\rhobar$ is totally odd;
  \item $\rhobar(G_F)=\GL_2(\F)$;
  \item if $v\nmid p$ is a place of $F$ then $\rhobar|_{G_{F_v}}$ is unramified;
  \item if $v|p$ is a place of $F$ then $\rhobar|_{G_{F_v}}\cong\rbar$;
  \item $[F:\Q]$ is even;
  \item $\rhobar$ is modular.
  \end{enumerate}
\end{prop}
\begin{proof}
  By Proposition 3.2 of \cite{frankII}, we may find $F$ and $\rhobar$
  satisfying all but the last two conditions. By Proposition 8.2.1 of
  \cite{0905.4266}, there is a finite Galois extension $F'/F$ in which all places above $p$ split completely such that $\rhobar|_{G_{F'}}$
  is modular. If we make a further quadratic extension (linearly disjoint from
  $\overline{F}^{\ker\rhobar}$ over $F$, and in which the primes above $p$ split completely),
  if necessary,
  to ensure that $[F':\Q]$ is even, and replace $F$ with $F'$, then the result follows.
\end{proof}

For the remainder of the subsection, we follow the arguments of \cite[\S 2.2]{kisinfmc}
very closely, and we do our best to conform to the notation used there.

We choose a finite set  $S$ of finite places of $F$, containing all the
places $v|p$ and at least one other place. Using 
\cite[Lem.~2.2.1]{kisinfmc}, we can and do choose $S$ so that conditions (1)--(4)
of \cite[\S 2.2]{kisinfmc} hold.

We denote by $D$ the quaternion
algebra with centre $F$ which is ramified at all infinite places of
$F$ and unramified at all finite places (so the set $\Sigma$
considered in \cite[\S 2.1.1]{kisinfmc} is empty). We fix a
maximal order $\cO_D$ of $D$, for each finite place $v$ we choose an
isomorphism $(\cO_D)_v\isoto M_2(\cO_{F_v})$, and we define  $U=\prod_v
U_v\subset(D\otimes_F\A_F^{\infty})^\times$ to be the following compact open subgroup
of $\prod_v(\cO_D)_v^\times$: if $v|p$ or $v\notin
S$ then $U_v=(\cO_D)^\times_v$, while if $v\in S$ but $v\nmid p$, then
$U_v$ consists of the matrices which are upper-triangular and
unipotent modulo $\varpi_v$, where $\varpi_v$ is a uniformiser of
$F_v$. The subgroup $U$ is sufficiently small in the sense of~\cite[\S 2.1.1]{kisinfmc}.

We write $\Sigma_p$ for the set of places $v|p$ of $F$ (since $\Sigma$ is
empty, this is consistent with \cite{kisinfmc}). For each $v|p$, we
let $R_v^{\square,\psi}(\rhobar|_{G_{F_v}})$ denote the universal framed deformation
$\cO$-algebra for $\rhobar|_{G_{F_v}}$ with determinant
$\psi\varepsilon$, and we define
\[\Rinfty:=\widehat{\otimes}_{v|p,\cO}R_v^{\square,\psi}(\rhobar|_{G_{F_v}})[[x_1,\dots,x_g]],\] 
where $x_1,\dots,x_g$ are formal variables, and the integer $g$ is chosen as
in~\cite[Prop.~2.2.4]{kisinfmc}.

For each place $v|p$ of $F$, we choose integers $a_v$, $b_v$ with
$b_v\ge~0$, together with an inertial type~$\tau_v$, and let $*$ be
either $\cris$ or nothing (the same choice of $*$ being made for all
$v|p$). We assume that
$\det\tau_v=\varepsilon^{1-2a_v-b_v}\psi|_{I_\Qp}$ for each~$v$. We
write\footnote{The
ring $\Rbarinfty$ depends upon 
the particular choices of $a_v$, $b_v$, $\tau_v$ and $*$, although (following \cite{kisinfmc}) we 
do not indicate this in the notation.} \[\Rbarinfty :=
\widehat{\otimes}_{v|p,\cO}R^{\square,\psi}_*(a_v,b_v,\tau_v,\rhobar|_{G_{F_v}})[[x_1,\ldots,x_g]].\]

If we write $W_\sigma:=\otimes_{v|p}L^*_{a_v,b_v,\tau_v}$, then $W_{\sigma}$ is
a finite free
$\cO$-module with an action of $\prod_{v|p}\GL_2(\cO_{F_v})$,
and the quotient $W_{\sigma}/\pi W_{\sigma}$ is then a representation 
of $\prod_{v|p}\GL_2(\cO_{F_v})$ over $\F$.
Fix a
Jordan--H\"older filtration
\[0=L_0\subset\dots\subset L_s =
W_{\sigma}/\pi W_{\sigma} \]
of $W_{\sigma}/\pi W_{\sigma}$
by $\prod_{v|p} \GL_2(\cO_{F_v})$-subrepresentations.
If we write $\sigma_i := L_i/L_{i-1}$,
then 
$\sigma_i\cong\otimes_{v|p}\det^{m_{v,i}}\otimes \Sym^{n_{v,i}}\F^2$
for some uniquely determined integers $m_{v,i} \in \{0, \ldots, p-2\}$ and
$n_{v,i}\in \{0,\ldots,p-1\}$.

The patching construction of
\cite[\S~2.2.5]{kisinfmc} then gives, for some integer denoted by $h+j$
in \cite{kisinfmc}, and formal variables $y_1,\dots,y_{h+j}$,
\begin{itemize}
\item an $(\Rbarinfty,\cO[[y_1,\dots,y_{h+j}]])$-bimodule
  $M_\infty$, finite free over $\cO[[y_1,\dots,y_{h+j}]]$, and
\item a filtration of $M_{\infty}/\pi M_{\infty}$ by $(\Rbarinfty/\pi,\F[[y_1,\dots,y_{h+j}]])$-bimodules,
say \[0=M_\infty^0\subset M_\infty^1\subset\dots\subset
  M_\infty^s=M_\infty/\pi M_{\infty},\] such that each
  $M^i_\infty/M^{i-1}_\infty$ is a finite free
  $\F[[y_1,\dots,y_{h+j}]]$-module, and such that
\item the isomorphism class of $M^i_\infty/M^{i-1}_\infty$ as an
  $(\Rinfty,\F[[y_1,\dots,y_{h+j}]])$-bimodule depends only on
  the isomorphism class of $\sigma_i$ as a
  $\prod_{v|p}\GL_2(\cO_{F_v})$-module, and not on  the choices of
  $a_v$, $b_v$ and $\tau_v$. (It
  is immediate from the finiteness argument used in patching that this
  can be achieved for any finite collection of tuples $(a_v,b_v,\tau_v)$, and
  since there are only countably many tuples, a diagonalization
  argument allows us to assume independence for all tuples.)
\end{itemize}

For any Serre weight $\sigma_{m,n}$, we now write $\psi_\cris$ for a
crystalline character lifting $\omega^{-1}\det\rbar$, and we
define (in the notation of Remark~\ref{rem: mu is determined by low weight crystalline}) \[\mu_{m,n}(\rbar) := e(R_\cris^{\square,\psi_\cris}(\widetilde{m},n,\triv,\rbar)/\pi)\]
and \[\cC_{m,n} := \cycle(R_\cris^{\square,\psi_\cris}(\widetilde{m},n,\triv,\rbar)/\pi).\]

\medskip

\noindent{{\em Proof of Theorem}~\ref{thm: the refined GL2 conjecture
    holds}.} Fix $a$, $b$, $\tau$ and $*$, and set $a_v=a$, $b_v=b$,
$\tau_v=\tau$ for each $v|p$. Lemma~2.2.11 of
\cite{kisinfmc} (together with its proof) shows that $e(\Rbarinfty/\pi)\ge e(M_\infty/\pi
    M_\infty,\Rbarinfty/\pi)$, and the following conditions are equivalent.
  \begin{enumerate}
  \item $M_\infty$ is a faithful $\Rbarinfty$-module.
  \item $M_\infty$ is a faithful $\Rbarinfty$-module which has rank
    $1$ at all generic points of $\Rbarinfty$.\footnote{We remark that
there is a typo in this part of the statement of \cite[Lem.~2.2.11]{kisinfmc};
$R_{\infty}$ is written there, rather than $\Rbar_{\infty}$.}
  \item $e(\Rbarinfty/\pi)=e(M_\infty/\pi
    M_\infty,\Rbarinfty/\pi)$.
\item $e(\Rbarinfty/\pi)\le e(M_\infty/\pi
    M_\infty,\Rbarinfty/\pi)$.
  \end{enumerate}
(Note that the proof of \cite[Lem.~2.2.11]{kisinfmc}, as written,
literally applies only when $a = 0$,
but in fact it goes through unchanged for any value of $a$.)
Furthermore, the argument of Corollary
2.2.17 of \cite{kisinfmc} (which uses the $p$-adic Langlands
correspondence for $\GL_2(\Qp)$, together with the weight part of Serre's conjecture)
establishes the inequality $e\bigl(\Rbarinfty/\pi)\le e(M_\infty/\pi
M_\infty,\Rbarinfty/\pi\bigr)$, so that in fact all of the conditions~(1)--(4) actually hold.  (This is where our hypothesis regarding $\rbar$ is used.)

 Let
  $\bar{R}_\infty^i:=
\bigl(\cotimes_{v|p}R_\cris^{\square,\psi_\cris}(\widetilde{m}_{v,i},n_{v,i},\triv,\rhobar|_{G_{F_v}})/\pi\bigr)
[[x_1,\ldots,x_g]],$
regarded as a quotient of $\Rinfty$. Since the isomorphism class of
  $M^i_\infty/M^{i-1}_\infty$ as an $\Rinfty$-module depends only
  on the isomorphism class of $\sigma_i$, one sees as in the proof of
  Lemma 2.2.13 of \cite{kisinfmc} that the action of $\Rinfty$ on
  $M^i_\infty/M^{i-1}_\infty$ factors through $\bar{R}_\infty^i$. (In
  brief: it is enough to check for each place $v_0|p$ that the action
  of $R_{v_0}^{\square,\psi}/\pi$ factors through
  $R_\cris^{\square,\psi_\cris}(\widetilde{m}_{v_0,i},n_{v_0,i},\triv,\rhobar|_{G_{F_{v_0}}})/\pi$. In
  order to do this, one applies the construction with $a_v=\widetilde{m}_{v_0,i}$,
  $b_v=n_{v_0,i}$, $\tau_{v_0}$ an appropriate scalar type, and the other $\tau_v$ chosen so that
  $\sigma_{m_{v,i},n_{v,i}}$ is a Jordan--H\"older factor of
  $L_{a_v,b_v,\tau_v}\otimes_\cO\F$ for $v\ne v_0$.)
Furthermore, the $\Rinfty$-module $M^i_\infty/M^{i-1}_\infty$ is
supported on all of $\Spec \bar{R}_\infty^i$, by the results of Section 4.6 of
\cite{gee061} (cf.\ the final paragraph of the proof of~\cite[Prop.~2.2.15]{kisinfmc}). 
Finally, we note that $\bar{R}_\infty^i$ is generically
reduced (see the proof of~\cite[Prop.~2.2.15]{kisinfmc}).

Lemma~\ref{lem:product cycles} shows that
\ssinc
\begin{equation}
\label{eqn:cycle equality}
\cycle(\Rbarinfty/\pi)=\Bigl(\prod_{v|p}\cycle(R_*^{\square,\psi}(a,b,\tau,\rbar)/\pi)\Bigr)
\times \cycle(\Spec \F[[x_1,\ldots,x_g]]),
\end{equation}
while,
identifying each of the modules $M_{\infty}/\pi M_{\infty}$, $M^i_{\infty}/M^{i-1}_\infty$, etc., with
the corresponding sheaves on $\Spec \Rinfty/\pi$ that they give rise to,
we compute that
\ssinc
\begin{multline}
\label{eqn:cycle inequality}
\cycle(M_\infty/\pi M_\infty)=
\sum_i \cycle(M^i_\infty/M^{i-1}_\infty) \\
\ge\sum_i \cycle(\Spec \Rbar^i_\infty/\pi)
=\Bigl(\prod_{v|p}\sum_{m,n} a^*_{m,n}\cC_{m,n}\Bigr)\times \cycle(\Spec \F[[x_1,\ldots,x_g]]),
\end{multline}
the first equality following 
by Lemma~\ref{lem:cycles in exact sequences},
the inequality following from the fact, noted above, that the support of
$M^i_{\infty}/M^{i-1}_{\infty}$ coincides with the generically reduced
closed subscheme $\Spec \bar{R}_\infty^i$ of $\Spec \Rinfty/\pi,$
and the second equality following from another application of Lemma~\ref{lem:product cycles}.
Also, since $M_{\infty}$ is $\pi$-torsion free,
and generically free of rank one
over each component of $\Spec \Rbarinfty$ (condition~(2) above),
Proposition~\ref{prop:cutting out} shows that
\ssinc
\begin{equation}
\label{eqn:another cycle equality}
\cycle(\Rbarinfty/\pi)  = \cycle(M_\infty/\pi M_\infty).
\end{equation}

Putting these computations together, we find that 
\begin{multline*}
\Bigl(\prod_{v|p}\cycle(R_*^{\square,\psi}(a,b,\tau,\rbar)/\pi)\Bigr)
\times \cycle(\Spec \F[[x_1,\ldots,x_g]])  \buildrel\text{(\ref{eqn:cycle equality})} \over =
\cycle(\Rbarinfty/\pi) \\
\buildrel \text{(\ref{eqn:another cycle equality})} \over =
\cycle(M_\infty/\pi M_\infty,\Rbarinfty/\pi) 
\buildrel \text{(\ref{eqn:cycle inequality})} \over \ge
\Bigl(\prod_{v|p}\sum_{m,n}a^*_{m,n}\cC_{m,n}\Bigr) \times \cycle(\Spec \F[[x_1,\ldots,x_g]]).
\end{multline*}
However, if we apply Lemma \ref{lem:multiplicities from cycles} (with
$y$ the closed point of $\Spec\Rbarinfty/\pi$) to pass to
the corresponding multiplicities, then this inequality on cycles gives
a corresponding inequality on multiplicities, which is in fact an
equality, by Lemma 2.3.1 of~\cite{kisinfmc}.  Thus this inequality of cycles is
an equality (note that any non-zero cycle must have non-zero
multiplicity at the unique closed point of $\Spec\Rbarinfty/\pi$), and
we deduce that
\[\prod_{v|p}\cycle(R_*^{\square,\psi}(a,b,\tau,\rbar)/\pi)=
\prod_{v|p}\sum_{m,n}a^*_{m,n}\cC_{m,n},\]
and thus that
\[\cycle(R_*^{\square,\psi}(a,b,\tau,\rbar)/\pi)=
\sum_{m,n}a^*_{m,n}\cC_{m,n},\]
as required.
\qed
\begin{remark}\label{rem: why we don't patch in the 1 by omega case}
  {\em Even if $\rbar\sim
    \begin{pmatrix}
      \omega\chi&*\\0&\chi
    \end{pmatrix}$ for some $\chi$, it is presumably possible to use
    the above arguments to show that our geometric Breuil--M\'ezard
    conjecture is equivalent to the usual one, and that both are
    equivalent to the equivalent conditions of Lemma 2.2.11 of
    \cite{kisinfmc}. We have not done so, because parity issues with
    Hilbert modular forms make the argument rather longer than one
    would wish, and in any case we prove a similar statement in far
    greater generality in Section \ref{sec: patching for
  GLn} (see Theorem \ref{thm: purely local statement assuming that for Serre weights,
    and in addition the statement of the support for lambda and
    tau}, which together with Lemma \ref{unrestricted determinant lifting ring is smooth over the
    fixed determinant ring} shows the equivalence of our geometric
  conjecture with the usual one). (Note that since $p>2$ all the Serre weights occurring in
  the reduction mod $p$ of $L_{a,b,\tau}$ have the same parity, so it
  is possible to circumvent parity problems by twisting, but the
  details are a little unpleasant to write out.)}
\end{remark}

\subsection{Analysis of components}\label{subsec: components for GL2
  Qp}For a given $\rbar$, there may be several different Serre weights
$\sigma_{m,n}$ for which $\mu_{m,n}(\rbar)\ne 0$. We now examine the
different cycles $\cC_{m,n}$. 

We write $W(\rbar)$ for
the set of $\sigma_{m,n}$ for which $\mu_{m,n}(\rbar)\ne 0$ (that is:
the set of Serre weights for $\rbar$).
The only cases where $W(\rbar)$ contains more than one element are
as follows (throughout it is understood that we always impose the 
conditions that $0\leq m \leq p-2$ and $0 \leq n \leq p-1$):
\begin{itemize}
\item If $\rbar$ is irreducible, say 
$$\rbar|_{I_\Qp}\sim\omega^m\otimes
  \begin{pmatrix}
    \omega_2^{n+1}&0\\0&\omega_2^{p(n+1)}
  \end{pmatrix},$$
then $W(\rbar)=\{\sigma_{m,n},\sigma_{m+n,p-1-n}\}$.
\item If $\rbar$ is reducible but indecomposable, with  
  $$\rbar|_{I_\Qp}\sim 
  \begin{pmatrix}\omega^{m+1}&*\\0&\omega^m
  \end{pmatrix}$$
a peu ramifi\'ee extension, then $W(\rbar)=\{\sigma_{m,0},\sigma_{m,p-1}\}$.
\item If $\rbar$ is reducible and decomposable, with
  $$\rbar|_{I_\Qp}\sim 
  \begin{pmatrix}\omega^{m+n+1}&0\\0&\omega^m
  \end{pmatrix}$$ such that $n\le p-3$, then
  \begin{itemize}
  \item if $0<n<p-3$, we have $W(\rbar)=\{\sigma_{m,n},\sigma_{m+n+1,p-3-n}\}$;
  \item if $n=p-3$ and $p>3$,
we have $W(\rbar)=\{\sigma_{m,p-3},\sigma_{m-1,0},\sigma_{m-1,p-1}\}$;
  \item if $p=3$ and $n=0$, we have
$W(\rbar)=\{\sigma_{m,0},\sigma_{m,2},\sigma_{m+1,0},\sigma_{m+1,2}\}$.
  \end{itemize}
\end{itemize}
The relationships between the cycles $\cC_{m,n}$ are as follows. Note
in particular that by Theorem \ref{thm: the refined GL2 conjecture
  holds}, if  $\rbar\not\sim
    \begin{pmatrix}
      \omega\chi&*\\0&\chi
    \end{pmatrix}$ for any $\chi$, then the assumption in the
    following Proposition is automatic.
\begin{prop}
  \label{prop:relationship between cycles for GL2 Qp} Assume that
  Conjecture~{\em \ref{conj: refined BM conjecture for GL2 Qp}} holds for
  $\rbar$. Then the cycles $\cC_{m,n}$ have disjoint support, except
  for the cycles $\cC_{m,0}$ and $\cC_{m,p-1}$ when both are
  non-zero. In this latter case there is an equality
  $\cC_{m,0}=\cC_{m,p-1}$, except if $\rbar$ is a twist of a {\em (}possibly split{\em )}
peu ramifi\'ee extension of the trivial character by the mod $p$
cyclotomic character, 
in which case
  $\cC_{m,p-1}$ is the sum of two irreducible cycles, one of
  which is~$\cC_{m,0}$.
\end{prop}
\begin{proof}
  It is presumably possible to establish this in most cases via direct
  computations with Fontaine-Laffaille theory; however, we take the
  opportunity to use our geometric formulation of the Breuil--Mezard conjecture.
  We begin by noting that it follows from Corollary 1.7.14 of
  \cite{kisinfmc} that the cycles $\cC_{m,n}$ are irreducible,
  except in the cases that $n=p-1$ and $\rbar|_{I_\Qp}\sim
  \begin{pmatrix}\omega^{m+1}&*\\0&\omega^m
    
  \end{pmatrix}$ is peu ramifi\'ee, or $n=p-2$ and $\rbar\sim\omega^m\otimes
  \begin{pmatrix}
    \mu_1&0\\0&\mu_2
  \end{pmatrix}$ for distinct unramified characters $\mu_1$ and
  $\mu_2$. Thus to prove the claimed disjointness of cycles, it is
  enough to prove that the cycles are not equal.

  We first consider the cases where $\cC_{m,0}$ and
  $\cC_{m,p-1}$ are nonzero.
  For notational simplicity, we make a twist by $\omega^{-m}$, and 
  thus further assume that $m = 0$.
  Consider the trivial inertial
  type $\triv$. Examining Henniart's appendix to \cite{breuil-mezard}, we
  see that $\sigma(\triv)=\St$, the Steinberg representation, while
  $\sigma^\cris(\triv)=\triv$, the trivial representation. The reduction mod
  $p$ of $\St$ is just $\sigma_{p-1}$, so we conclude
  that
\[\cC_{0,p-1}=\cycle(R^{\square,\psi_\cris}(\widetilde{0},0,\triv,\rbar)/\pi),\] \[\cC_{0,0}=\cycle(R^{\square,\psi_\cris}_\cris(\widetilde{0},0,\triv,\rbar)/\pi).\]
  By definition, $R^{\square,\psi_\cris}_\cris(\widetilde{0},0,\triv,\rbar)$ is a quotient of
  $R^{\square,\psi}(0,0,\triv,\rbar)$, and as they have the same dimension,
  we see that $\Spec R_{\cris}^{\square,\psi_\cris}(\widetilde{0},0,\triv,\rbar)$ is a union of irreducible components of
  $\Spec R^{\square,\psi_\cris}(\widetilde{0},0,\triv,\rbar)$.

  It is easy to see that the two rings
  are actually equal unless $\rbar$ is a twist of a peu ramifi\'e extension of
  $1$ by $\omega$, because the only non-crystalline semistable
  crystalline representations with Hodge--Tate weights $0$ and $1$ are
  unramified twists of extensions of $1$ by $\varepsilon$, and so it suffices to
  check, in the case when  $\rbar$ {\em is} a twist of a peu ramifi\'e extension of
  $1$ by $\omega$, that $\Spec R^{\square,\psi_\cris}(\widetilde{0},0,\triv,\rbar)/\pi)$ has two distinct
  components. In the case that the extension is non-split, the framed
  deformation ring is formally smooth over the deformation ring, and
  the relevant rings are computed in Theorem 5.3.1(i) of \cite{breuil-mezard};
  in particular, they do verify that the $\Spec$s of their reductions mod $\pi$ contain two
  distinct components.

  We now explain another way to see that
$\Spec R^{\square,\psi_\cris}(\widetilde{0},0,\triv,\rbar)/\pi)$
  consists of two distinct components,
  which works equally well in the case when $\rbar$ is split.
  Namely, a two-dimensional crystalline representation with Hodge--Tate weights $0$ and $1$
  with reducible reduction is necessarily an extension of characters,
  whose restrictions to
  $I_p$ are trivial and cyclotomic respectively, and it is uniquely determined by its
  associated pseudo-representation.  The cycle $\cC_{0,0}$
  is thus directly seen to be irreducible,
  and it has non-trivial image in the associated pseudo-deformation space.
  On the other hand, as we already observed, a genuinely semi-stable two-dimensional
  deformation of $\rbar$ with Hodge--Tate weights $0$ and $1$ is necessarily a twist
  of an extension of the trivial character by the cyclotomic character, with the possible twist
  being uniquely determined (since we have fixed the determinant to be $\psi \varepsilon$).
  Such extensions are determined by their $\mathcal L$-invariant, and so one can give
  an explicit description of the Zariski closure of the space of genuinely semi-stable
  deformations, and show that this closure, as well as its reduction mod $\pi$, is irreducible.
  Moreover, its reduction $\bmod \, {\pi}$ does not coincide with $\cC_{0,0}$,
  since its image in the associated pseudo-deformation space is simply the closed point.
  Thus we have shown that $\cC_{0,p-1}$ is the sum of two distinct irreducible components.
 
Consider next the case that $\rbar|_{I_\Qp}\sim\omega^m\otimes
  \begin{pmatrix}
    \omega_2^{n+1}&0\\0&\omega_2^{p(n+1)}
  \end{pmatrix}$ with $0<n<p-1$. We need to show that
  $\cC_{m,n}\ne\cC_{m+n,p-1-n}$. Consider the inertial
  type $\omegat^{m+n}\oplus\omegat^m$, and choose $\psi$ so that $\psi|_{I_\Qp}=\varepsilon\omegat^{2m+n}$. Then by for example Lemmas
  3.1.1 and 4.2.4 of \cite{MR1639612}, the
  semisimplification of the reduction modulo $p$ of
  $\sigma(\omegat^{m+n}\oplus\omegat^m)$ (the representation of
  $\GL_2(\Zp)$ associated to $\omegat^{m+n}\oplus\omegat^m$ by
  Henniart's inertial local Langlands correspondence, which is the
  inflation to $\GL_2(\Zp)$ of a principal series representation of $\GL_2(\Fp)$)  has
  Jordan--H\"older factors $\sigma_{m,n}$ and $\sigma_{m+n,p-1-n}$, so
  that we
  have \[\cycle(R^{\square,\psi}(0,1,\omegat^{m+n}\oplus\omegat^m,\rbar)/\pi)=\cC_{m,n}+\cC_{m+n,p-1-n}.\]
  By Theorem 6.22 of \cite{MR2137952} (noting that the framed
  deformation ring is formally smooth over the deformation ring, since
  $\rbar$ is irreducible), we have
  $R^{\square,\psi}(0,1,\omegat^{m+n}\oplus\omegat^m,\rbar)/\pi\cong\F[[U,V,W,X,Y]]/(XY)$, so
  $\cC_{m,n}\ne\cC_{m+n,p-1-n}$, as required.

  It remains to treat the cases where
  $\rbar\sim\chibar_1\oplus\chibar_2$ is a direct sum of two distinct
  characters. In this case, one knows that all of the relevant lifting
  rings $R^{\square,\psi_\cris}_\cris(\widetilde{m},n,\triv,\rbar)$ are in fact ordinary
  lifting rings (cf. Corollary 1.7.14 and Remark 1.7.16 of
  \cite{kisinfmc}). It thus suffices to note that the cycles that we
  have to prove are distinct correspond respectively to liftings which
  contain a submodule lifting $\chibar_1$ or which contain a submodule
  lifting $\chibar_2$, and since $\chibar_1\ne\chibar_2$, it is easy
  to see that the cycles are distinct.
\end{proof}
\noindent {\em Proof of Proposition~\ref{prop:components}.} The
assertions about the number of components of $\cC_{m,n}$ follow
from Proposition \ref{prop:relationship between cycles for GL2 Qp} above
together with Corollary 1.7.14 of \cite{kisinfmc}, and the
disjointness of the supports of the $\cC_{m,n}$ follows from
Proposition \ref{prop:relationship between cycles for GL2 Qp}.\qed

\subsection{Comparison with the results of \cite{breuilmezardIII}}
\label{subsec:comparison}
In \cite{breuilmezardIII}, Breuil and M\'ezard construct, for generic
$\rbar$, a correspondence between the irreducible components
of $\Spec R^{\square,\psi}(a,b,\tau,\rbar)/\pi$ and the
Serre weights for $\rbar$
which appear with positive multiplicity in the mod $p$ reduction
of $(\det^a\otimes \Sym^b E^2)\otimes_E \sigma(\tau)$. 
The definition of the correspondence is given in
\cite[Thm.~1.5]{breuilmezardIII}:
namely, if $\mathfrak a$ is a generic point
of $\Spec R^{\square,\psi}(a,b,\tau,\rbar)/\pi$,
then the associated Serre weight is the $\GL_2(\Z_p)$-socle
of a certain $\GL_2(\Q_p)$-representation obtained from the universal Galois representation into $\GL_2\bigl(R^{\square,\psi}
(a,b,\tau,\rbar)/(\pi,\mathfrak a)\bigr)$
via the $p$-adic local Langlands correspondence.

According to our geometric formulation of the Breuil--M\'ezard conjecture,
the component $\overline{\{\mathfrak a\}}$
of $\Spec R^{\square,\psi}(a,b,\tau,\rbar)/\pi$
is equal to
$\Spec R^{\square,\psi_\cris}_{\cris}(\widetilde{m},n,\triv,\rbar)/\pi)$,
for some Serre weight $\sigma_{m,n}$ of $\rbar$.
Since the $p$-adic local Langlands correspondence is functorial, 
we find that 
the $\GL_2(\Z_p)$-socle of the $\GL_2(\Q_p)$-representation attached to $\mathfrak a$,
thought of as {\em a} generic point of 
$\Spec R^{\square,\psi}(a,b,\tau,\rbar)/\pi$,
is the same as 
the $\GL_2(\Z_p)$-socle of the $\GL_2(\Q_p)$-representation
attached to $\mathfrak a$,
thought of as {\em the} generic point of 
$\Spec R^{\square,\psi_\cris}_{\cris}(\widetilde{m},n,\triv,\rbar)/\pi)$.
But in this latter case, the $\GL_2(\Z_p)$-socle in question
is equal to $\sigma_{m,n}$, as follows from the fact that
the correspondence of \cite{breuilmezardIII} is compatible with
the original conjecture of \cite{breuil-mezard}.
This shows that the correspondence of \cite{breuilmezardIII} is 
precisely the correspondence $\cC_{m,n} \mapsto \sigma_{m,n}$.

\section{The Breuil--M\'ezard Conjecture for $\GL_n$}\label{sec: BM
  conjecture for GLn}\subsection{The numerical conjecture}We begin by formulating a natural generalisation
of the Breuil--M\'ezard conjecture for $n$-dimensional
representations. We now fix the notation we will use for the rest of
the paper, which differs in some respects from that of Section
\ref{sec:GL2 Qp}, but is closer to that used in the literature on
automorphy lifting theorems for unitary groups. We remind the reader
that for the rest of the paper, our convention on Hodge--Tate weights
is that the Hodge--Tate weight of $\varepsilon$ is $-1$. Let $K/\Qp$ be a
finite extension with ring of integers $\cO_K$ and residue field $k$,
let $E/\Qp$ be a finite extension with ring of integers $\cO$,
uniformiser $\pi$ and residue field $\F$, and let
$\rbar:G_K\to\GL_n(\F)$ be a continuous representation. Assume that
$E$ is sufficiently large, and in particular that $E$ contains the
images of all embeddings $K\into\Qpbar$.
Let $\Z^n_+$ denote the set of tuples $(\lambda_1,\dots,\lambda_n)$ of
integers with $\lambda_1\ge \lambda_2\ge\dots\ge \lambda_n$. For any $\lambda\in\Z^n_+$, view
$\lambda$ as a dominant character of the algebraic group $\GL_{n/\cO}$ in
the usual way, and
let $M'_\lambda$ be the algebraic $\cO_K$-representation of $\GL_n$ given
by \[M'_\lambda:=\Ind_{B_n}^{\GL_n}(w_0\lambda)_{/\cO_K}\] where $B_n$ is the
 Borel subgroup of upper-triangular matrices of $\GL_n$, and $w_0$ is the longest element
of the Weyl group (see \cite{MR2015057} for more details of these
notions, and note that $M'_\lambda$ has highest weight $\lambda$). Write $M_\lambda$ for the $\cO_K$-representation of $\GL_n(\cO_K)$
obtained by evaluating $M'_\lambda$ on $\cO_K$. For any
$\lambda\in(\Z^n_+)^{\Hom_\Qp(K,E)}$ we write $L_\lambda$ for the
$\cO$-representation of $\GL_n(\cO_K)$ defined
by \[\otimes_{\tau:K\into
  E}M_{\lambda_\tau}\otimes_{\cO_K,\tau}\cO.\]
Given any $a\in\Z^n_+$ with $p-1\ge a_i-a_{i+1}$ for all $1\le i\le
n-1$, we define the $k$-representation $P_a$ of $\GL_n(k)$ to be the
representation obtained by evaluating $\Ind_{B_n}^{\GL_n}(w_0 a)_{k}$
on $k$, and let $N_a$ be the irreducible sub-$k$-representation of $P_a$
generated by the highest weight vector (that this is indeed
irreducible follows for example from II.2.8(1) of \cite{MR2015057} and
the appendix to \cite{herzigthesis}).
 We say that an element $a\in(\Z^n_+)^{\Hom(k,\F)}$ is a \emph{Serre
  weight} if
\begin{itemize}
\item for each $\sigma\in\Hom(k,\F)$ and each $1\le i\le n-1$ we
have \[p-1\ge a_{\sigma,i}-a_{\sigma,i+1},\]
\item and for each $\sigma$ we have $0\le a_{\sigma,n}\le p-1$, and
  not all $a_{\sigma,n}=p-1$.
\end{itemize}
If $a\in(\Z^n_+)^{\Hom(k,\F)}$ is a Serre weight then we define an
irreducible $\F$-representation $F_a$ of $\GL_n(k)$
by \[F_a:=\otimes_{\tau\in\Hom(k,\F)}N_{a_\tau}\otimes_{k,\tau}\F.\]
The representations $F_a$ are absolutely irreducible and pairwise
non-isomorphic, and every irreducible $\F$-representation of $\GL_n(k)$ is
of the form $F_a$ for some $a$ (see for example the appendix to
\cite{herzigthesis}).
We say that  an
element $\lambda\in(\Z^n_+)^{\Hom_\Qp(K,E)}$ is a \emph{lift} of a
Serre weight $a\in(\Z^n_+)^{\Hom(k,\F)}$ if for each $\sigma\in\Hom(k,\F)$ there is an element
$\tau\in\Hom_\Qp(K,E)$ lifting $\sigma$ such
that $\lambda_{\tau}=a_\sigma$, and for all other
$\tau'\in\Hom_\Qp(K,E)$ lifting $\sigma$ we have
$\lambda_{\tau'}=0$.
We have a partial ordering $\le$ on Serre weights, where $b\le 
a$ if and only if $a-b$ is a sum of (positive) simple roots.
\begin{lem}
  \label{lem: decomposition of L_lambda mod p}If $\lambda$ is a lift of $a$ then
$L_\lambda\otimes_\cO\F$ has socle $F_a$, and every other
Jordan--H\"older factor of $L_\lambda\otimes_\cO\F$ is of the form
$F_b$ with $b<a$.
\end{lem}
\begin{proof}
 This follows from sections 5.8 and 5.9 of \cite{MR2199819} (noting
 that the orderings $\le$ and $\le_\Q$ coincide for $\GL_n$).
\end{proof}
Let $\tau:I_K\to\GL_n(E)$ be a representation with open kernel which
extends to $W_K$, and take $\lambda\in(\Z^n_+)^{\Hom_\Qp(K,E)}$. Let
$\rbar:G_K\to\GL_n(\F)$ be a continuous representation. If $E'/E$ is
a finite extension, we say that a potentially crystalline
representation $\rho:G_K\to\GL_n(E')$ has \emph{Hodge type} $\lambda$
if for each $\tau:K\into
E$, \[\HT_{\tau}(\rho)=\{\lambda_{\tau,1}+n-1,\lambda_{\tau,2}+n-2,\dots,\lambda_{\tau,n}\}.\]
We say that $\rho$ has \emph{inertial type} $\tau$ if the restriction
to $I_K$ of the Weil--Deligne representation associated to $\rho$ is
equivalent to $\tau$.
\begin{prop}\label{prop: existence of local deformation rings}
  For each $\lambda$, $\tau$ there is a unique quotient
  $R^\square_{\rbar,\lambda,\tau}$ of the universal lifting
  $\cO$-algebra $R^\square_{\rbar}$ for
  $\rbar$ with the following properties.
  \begin{enumerate}
  \item $R^\square_{\rbar,\lambda,\tau}$ is reduced and $p$-torsion free, and
    $R^\square_{\rbar,\lambda,\tau}[1/p]$ is formally smooth
    and equidimensional of dimension $n^2+[K:\Qp]n(n-1)/2$.
  \item If $E'/E$ is a finite extension, then an $\cO$-algebra homomorphism
    $R^\square_{\rbar}\to E'$ factors through
    $R^\square_{\rbar,\lambda,\tau}$ if and only if the
    corresponding representation $G_K\to\GL_n(E')$ is potentially
    crystalline of Hodge type $\lambda$ and inertial type $\tau$.
  \item $R^\square_{\rbar,\lambda,\tau}/\pi$ is equidimensional.
  \end{enumerate}
\end{prop}
\begin{proof}
  All but the final point are proved in \cite{kisindefrings}, and the
  final statement is a straightforward consequence of the first
(cf.\ the proof of Lemme 2.1 of \cite{breuilmezardIII}).
\end{proof}If $\tau$ is trivial we will write
$R^{\square}_{\rbar,\lambda}$ for
$R^{\square}_{\rbar,\lambda,\tau}$. The Breuil--M\'ezard conjecture
predicts the Hilbert--Samuel multiplicity
$e(R^\square_{\rbar,\lambda,\tau}/\pi)$. In order to state the
conjecture, it is first necessary to make a conjecture on the
existence of an inertial local Langlands correspondence for
$\GL_n$. The following is a folklore conjecture.
\begin{conj}\label{conj: inertial local Langlands, N=0}If $\tau$ is an
  inertial type, then there is a finite-dimensional smooth irreducible
  $\Qpbar$-representation $\sigma(\tau)$ of $\GL_n(\cO_K)$ such that
  if $\taut$ is any Frobenius-semisimple Weil--Deligne representation of $W_K$
over $\Qpbar$, then the restriction of
  $(\rec_p^{-1}(\taut)^\vee)$ to ${\GL_n(\cO_K)}$ contains
{\em (}an isomorphic copy of{\em )}
$\sigma(\tau)$ as a subrepresentation if and
  only if $\taut|_{I_K}\sim\tau$ and $N=0$ on $\taut$. If $p>n$ then
  $\sigma(\tau)$ is unique up to isomorphism.
\end{conj}\begin{remark}
{\em
  A very similar conjecture is formulated in
  \cite{conley10}, which also proves some partial results in the case
  $n=3$. We have formulated this conjecture only for representations
  with $N=0$, as we will only formulate our generalisations of the
  Breuil--M\'ezard conjecture for potentially crystalline
  representations, in order to avoid complications in the global
  arguments of Section \ref{sec: patching for GLn}. We expect that a
  ``semistable'' version of the conjecture will also be valid, where
  one removes the conclusion that $N=0$ on $\tau$, but adds the
  requirement that $\rec_p^{-1}(\taut)\otimes|\det|^{(n-1)/2}$ be
  generic.}
\end{remark}
Conjecture \ref{conj: inertial local Langlands, N=0} is proved in the
case $n=2$ by Henniart in the appendix to \cite{breuil-mezard}. It has
been proved for any $n$ for supercuspidal representations by Pa\v{s}k\={u}nas
(\cite{MR2180458}), but to the best of our knowledge it is open in
general. However, the important point for us is the existence of
$\sigma(\tau)$, rather than its uniqueness, and this is known in
general. The following is a special case of Proposition 6.5.3
of~\cite{MR2656025} (see also \cite{conleytypes}).

\begin{thm}\label{thm: inertial local Langlands, N=0}If $\tau$ is an
  inertial type, then there is a finite-dimensional smooth irreducible
  $\Qpbar$-representation $\sigma(\tau)$ of $\GL_n(\cO_K)$ such that
  if $\taut$ is any pure Frobenius-semisimple Weil--Deligne representation of $W_K$
  over $\Qpbar$, then the restriction of
  $(\rec_p^{-1}(\taut)\otimes|\det|^{(n-1)/2})$ to ${\GL_n(\cO_K)}$ contains
{\em (}an isomorphic copy of{\em )}
$\sigma(\tau)$ as a subrepresentation if and
  only if $\taut|_{I_K}\sim\tau$ and $N=0$ on $\taut$.
  \end{thm}

Enlarging $E$ if necessary, we may assume that $\sigma(\tau)$ is defined
over $E$. Since it is a finite-dimensional representation of the
compact group $\GL_n(\cO_K)$, it contains a $\GL_n(\cO_K)$-stable
$\cO$-lattice $L_\tau$.
Set $L_{\lambda,\tau}:=L_\tau\otimes_\cO L_\lambda$, a finite free
$\cO$-module with an action of $\GL_n(\cO_K)$. Then we may
write \[(L_{\lambda,\tau}\otimes_\cO\F)^{\semis}\isoto\oplus_a
F_a^{n_a},\]where the sum runs over the Serre weights
$a\in(\Z^n_+)^{\Hom(k,\F)}$, and the $n_a$ are nonnegative
integers. Then the generalised Breuil--M\'ezard conjecture is the following.
\begin{conj}
  \label{conj: generalised BM conjecture}There exist integers
  $\mu_a(\rbar)$ depending only on $\rbar$ and $a$ such that $e(R^\square_{\rbar,\lambda,\tau}/\pi)=\sum_an_a\mu_a(\rbar)$.
\end{conj}
\begin{remark}\label{rem: generalised BM conjecture remarks}
{\em
  \begin{enumerate}
  \item Assuming the conjecture, the integers $\mu_a(\rbar)$ may be
    computed recursively as follows. Let $\lambda_a$ be a lift of
    $a$. If $a$ is in the lowest alcove, then
    $L_\lambda\otimes_\cO\F=F_a$, and we have
    $\mu_a(\rbar)=e(R^{\square}_{\rbar,\lambda_a}/\pi)$. In
    general, by Lemma \ref{lem: decomposition of L_lambda mod p} we
    see that we may compute $\mu_a(\rbar)$ given
    $e(R^{\square}_{\rbar,\lambda_a}/\pi)$ and the values
    $\mu_b(\rbar)$, $b<a$. \item The reader might have expected the sum on the right hand side
    to only be over weights $a$ which are predicted Serre weights for
    $\rbar$. However, according to the philosophy explained in the
    introduction to~\cite{geekisin}, the predicted Serre weights $a$
    for $\rbar$ should be precisely the $a$ for which
    $\mu_a(\rbar)\ne 0$.
\item
The reader might wonder why we have formulated our $n$-dimensional
analogues of the 
Breuil--M\'ezard conjectures for liftings without fixed determinant,
when the conjecture is more usually stated as in Section \ref{sec:GL2
  Qp} for lifts with fixed determinant. The reason is that in the
global arguments we will make in Section \ref{sec: patching for GLn}
using unitary groups, it is the lifting rings without fixed
determinant that arise naturally. However, the conjectures with and
without fixed determinant are actually equivalent, at least as long as
$p>n$, as follows from Lemma~\ref{unrestricted determinant lifting ring is smooth
over the fixed determinant ring} below.
  \end{enumerate}
}
\end{remark}
\subsection{The geometric conjecture}\label{subsec: refined conjecture
  for GL_n}
We may now state our geometric conjecture, which is entirely analogous
to Conjecture \ref{conj: refined BM conjecture for GL2 Qp}.
\begin{conj}
  \label{conj: refined BM conjecture for GLn}For each Serre weight $a$ there is a cycle $\cC_a$ depending only
  on $\rbar$ and $a$ such
  that \[\cycle(R^\square_{\rbar,\lambda,\tau}/\pi)=\sum_{a}n_a\cC_a.\]
\end{conj}
\begin{remark}
{\em  Again, if one assumes the conjecture then one can inductively
  compute the cycles $\cC_a$ in terms of the cycles
  $\cycle(R^{\square}_{\rbar,\lambda_a}/\pi)$.}
\end{remark}
\subsection{Twisting by characters}\label{subsec: twisting by
  characters}We now establish a lemma which implies in particular that
the Breuil--M\'ezard conjecture for liftings without
a fixed determinant (as we have formulated it above) is equivalent
to the analogous conjecture for liftings with a fixed determinant
(at least if $p > n$).
Note that, for a given inertial type $\tau$ and Hodge type $\lambda,$
there is a character
$\psi_{\lambda,\tau}:I_K\to\cO^\times$ such that any lift $\rho$ of
$\rbar$ of Hodge type $\lambda$ and inertial type $\tau$ necessarily
has $\det\rho|_{I_K}=\psi_{\lambda,\tau}$. Let $\psi:G_K\to\cO^\times$ be a
character such that $\psibar=\det\rbar$ and
$\psi|_{I_K}=\psi_{\lambda,\tau}$. We let
$R_{\rbar,\lambda,\tau}^{\square,\psi}$
denote the quotient of
$R_{\rbar,\lambda,\tau}^{\square}$
corresponding to lifts with determinant $\psi$. In the case $n=2$ and
$K=\Qp$ this
is~\cite[2.2.2.9]{breuil-mezard}.
\begin{lem}
  \label{unrestricted determinant lifting ring is smooth over the
    fixed determinant ring}Suppose that $p>n$. Then
  $R_{\rbar,\lambda,\tau}^{\square}\cong
  R_{\rbar,\lambda,\tau}^{\square,\psi}[[X]]$.
\end{lem}
\begin{proof}Consider a finite extension $E'/E$ and a point
  $x:R_{\rbar,\lambda,\tau}^{\square}\to E'$ with corresponding
  representation $\rho'$. Let $\theta=\psi^{-1}\det\rho'$, so that
  $\theta$ is an unramified character with trivial reduction. Since
  $p>n$, we see that $\theta$ has an $\cO_{E'}$-valued $n$-th
  root. From this observation is it easy to see that if
  $\rho^\psi:G_K\to\GL_n(R_{\rbar,\lambda,\tau}^{\square,\psi})$ is
  the universal lifting of determinant $\psi$, and $\mu_x$ is the
  unramified character taking a Frobenius element to $x$, then
  $\rho^\psi\otimes\mu_{1+X}\circ\det:G_K\to\GL_n(R_{\rbar,\lambda,\tau}^{\square,\psi}[[X]])$
  is the universal lifting of arbitrary determinant, as required.
\end{proof}
\section{Global patching arguments}\label{sec: patching for
  GLn} Our goal in this section is to
employ the Taylor--Wiles--Kisin patching method so as to generalize, to
the extent possible, the results of Section~\ref{sec:GL2 Qp} to the
$n$-dimensional context. These arguments are essentially the natural
$n$-dimensional generalisation of the arguments of Section 4 of
\cite{geekisin}. We closely follow the approaches of \cite{kisinfmc}
and \cite{jack} (which in turn follows \cite{blgg} and \cite{cht}). In
particular, in the actual implementation of the patching method we
follow \cite{jack} very closely, although, because our ultimate
interests are local, we will sometimes make stronger global
assumptions in order to simplify the arguments; these stronger
assumptions can always be achieved in our applications.  Before
getting to the patching argument itself, we include a number of
preliminary subsections in which we briefly recall the necessary
background material on automorphic forms and Galois representations,
referring the reader to \cite{jack} for more details.

\subsection{Basic set-up}\label{subsec:basics}
We put ourselves in the setting of Section~\ref{sec: BM conjecture for
  GLn}, so that $K/\Qp$ is a finite extension, and
$\rbar:G_K\to\GL_n(\F)$ is a continuous representation. We also
assume from now on that $p>2$.

As in Section \ref{sec:GL2 Qp}, we begin by globalising $\rbar$.
Since the standard global context in which to study higher-dimensional
Galois representations is that of automorphic forms on unitary
groups, we briefly recall the various concepts that are required
to discuss Galois representations in that setting.

To begin with,
we recall from \cite{cht} that $\cG_n$ denotes the group scheme over $\Z$ defined to be the semidirect
product of $\GL_n\times\GL_1$ by the group $\{1,j\}$, which acts on
$\GL_n\times\GL_1$ by \[j(g,\mu)j^{-1}=(\mu\cdot{}^tg^{-1},\mu).\]We have a
homomorphism $\nu:\cG_n \to \GL_1,$ sending $(g,\mu)$ to $\mu$ and $j$
to $-1$. We refer the reader to Section 2.1 of \cite{cht} for a thorough discussion
of $\cG_n$, and of the relationship between $\cG_n$-valued
representations and essentially conjugate self-dual $\GL_n$-valued
representations.

\begin{terminology}
\label{term:restrictions}
{\em
To ease notation, we adopt the following convention with regard to Galois
representations with values in $\cG_n(\Fbar_p)$: if $F$ is an imaginary CM
field with maximal totally real subfield~$F^+$, and $\rhobar:G_{F^+}
\to \cG_n(\Fbar_p)$ is a continuous representation with $\rhobar(G_F)\subset\GL_n(\Fpbar)\times\GL_1(\Fpbar)$,
then we write $\rhobar|_{G_F}$ for
the restriction of $\rhobar$ to $G_F$, regarded as a representation
$G_F\to\GL_n(\Fpbar)$, and similarly for $\rhobar(G_{F(\zeta_p)})$ and
$\rhobar|_{G_{F_\tv}}$ (for places $\tv$ of $F$).
}
\end{terminology}

Recall that the notion of an \emph{adequate} subgroup of
$\GL_n(\Fpbar)$ is defined in \cite{jack}. We will not need the
details of the definition; some examples of representations whose
image is adequate will be constructed in Appendix~\ref{sec:local to
  global}.

We now state our basic hypothesis related to the globalization of $\rbar$.
Namely,
we assume that there is an imaginary CM field $F$ with
maximal totally real subfield~$F^+$, together with a continuous
representation $\rhobar:G_{F^+}\to\cG_n(\Fpbar)$, such that
\begin{itemize}
\item $F/F^+$ is unramified at all finite places,
\item $[F^+:\Q]$ is divisible by $4$,
\item every place $v|p$ of $F^+$ splits in $F$,
\item $\rhobar$ is automorphic in the sense of Definition \ref{defn:
    max ideal or mod p Galois rep is automorphic} below; in
  particular, $\nu\circ\rhobar=\varepsilonbar^{1-n}\delta_{F/F^+}^n$, where
  $\delta_{F/F^+}$ is the quadratic character corresponding to
  $F/F^+$,
\item $\rhobar^{-1}(\GL_n(\Fpbar)\times\GL_1(\Fpbar))=G_F$,
\item $\rhobar$ is unramified at primes $v\nmid p,$
\item $\rhobar(G_{F(\zeta_p)})$ is adequate, so that in particular
  $\rhobar|_{G_F}$ is irreducible,
\item $\overline{F}^{\ker\ad\rhobar|_{G_F}}$ does not contain
  $F(\zeta_p)$, and
\item for each place $v|p$ of $F^+$, there is a place $\tv$ of $F$
  lying over $v$ such that $F_{\tv}\cong K$ and $\rhobar|_{G_{F_\tv}}$
  is isomorphic to $\rbar$.
\end{itemize}
We say that such an ($F$ and) $\rhobar$ is a \emph{suitable globalization} of $\rbar$.

\begin{remark}
{\em
It will not always be the case that such a representation $\rhobar$
exists, because the assumption that $\rhobar(G_{F(\zeta_p)})$ is
adequate implies that $p\nmid n$. On the other hand, if we assume
that $p\nmid n$ and that Conjecture \ref{conj: existence of local
  crystalline lifts} holds for $\rbar$, then by Corollary \ref{cor: the final
local-to-global result} there is a suitable globalization of $\rbar$.
}
\end{remark}

We briefly explain the motivation for the various conditions that we require.
The first three ensure the existence of a convenient unitary group on which to work,
with the property that it is isomorphic to $\GL_n$ at places dividing~$p$.
The final condition ensures that $\rhobar$ can be used to study
$\rbar$. The remaining conditions are imposed in order to use patching
constructions of \cite{jack}; some of them are imposed in order to
simplify these constructions.  (When considering the last three conditions,
the reader should recall our terminological convention of~(\ref{term:restrictions}).)

\subsection{Unitary groups and algebraic automorphic forms}\label{sec: Unitary groups}
There is a reductive algebraic group $G/F^+$ with the following
properties (cf. Section 6 of \cite{jack}):
\begin{itemize}
\item $G$ is an outer form of $\GL_n$, with $G_{/F}\cong\GL_{n/F}$.
\item If $v$ is a finite place of $F^+$, $G$ is quasi-split at $v$.
\item If $v$ is an infinite place of $F^+$, then $G(F^+_v)\cong U_n(\R)$.
\end{itemize}
As in section 3.3 of \cite{cht} we
  may define a model for $G$ over
$\cO_{F^+}$. If $v$ is a place of $F^+$ which splits as $ww^c$ over
$F$, then we have an
isomorphism \[\iota_w:G(\cO_{F_v^+})\isoto \GL_n(\cO_{F_w}).\] 

Let $E/\Qp$ be a finite
extension with ring of integers $\cO$ and residue field $\F$, which we
assume is chosen to be sufficiently large that $\rhobar$ is valued in
$\cG_n(\F)$. From now on we will often regard $\rhobar|_{G_F}$
as being valued in $\GL_n(\F)$, and we write $V_\F$ for the underlying
$\F$-vector space of $\rhobar|_{G_F}$.

  Let $S_p$ denote the set of places of $F^+$ lying over $p$, and for
  each $v\in S_p$ fix a place $\tv$ of $F$ lying over $v$. Let
  $\tilde{S}_p$ denote the set of places $\tilde{v}$ for $v\in
  S_p$. Write $F^+_p:=F\otimes_\Q\Qp$, $\cO_{F^+_p}:=\cO_F\otimes_\Z\Zp$.
  Let $W$ be a finite $\cO$-module with an action of $G(\cO_{F^+_p})$,
  and let $U\subset G(\A_{F^+}^\infty)$ be a compact open subgroup
  with the property that for each $u\in U$, if $u_p$ denotes the
  projection of $u$ to $G(F_p^+)$, then $u_p\in G(\cO_{F^+_p})$. We
  say that $U$ is \emph{sufficiently small} if for all $t\in
  G(\A_{F^+}^\infty)$, $t^{-1}G(F^+)t\cap U$ does not contain an
  element of order $p$.  Let $S(U,W)$ denote the space of algebraic
  modular forms on $G$ of level $U$ and weight $W$, i.e. the space of
  functions \[f:G(F^+)\backslash G(\A_{F^+}^\infty)\to W\] with
  $f(gu)=u_p^{-1}f(g)$ for all $u\in U$. If $U$ is sufficiently small,
  then the functor $W\mapsto S(U,W)$ is exact.

For each $v \in S_p$ choose an inertial type $\tau_v:I_{F_v}\to\GL_n(E)$ and
a weight $\lambda_v\in(\Z^n_+)^{\Hom_\Qp(F_\tv,E)}$, and let
$L_{\lambda_v,\tau_v}$ be the $\cO$-representation of
$\GL_n(\cO_{F_\tv})$ defined in Section \ref{sec: BM
  conjecture for GLn}. Write $L_{\lambda,\tau}$ for the tensor product of
the $L_{\lambda_v,\tau_v}$, regarded as a representation of  $G(\cO_{F^+,p})$ by
letting  $G(\cO_{F^+,p})$ act on $L_{\lambda_v,\tau_v}$ via $\iota_\tv$, and for
any $\cO$-algebra $A$ we write $S_{\lambda,\tau}(U,A)$ for
$S(U,L_{\lambda,\tau}\otimes_\cO A)$.

\subsection{Hecke algebras and Galois representations}\label{sec: Galois repns}

This assumption made above that $\overline{F}^{\ker\ad\rhobar|_{G_F}}$
does not contain $F(\zeta_p)$ means that we can and do choose a finite
place $v_1\notin S_p$ of $F^+$ which splits over $F$ such that $v_1$
does not split completely in $F(\zeta_p)$, and
$\ad\rhobar|_{G_F}(\Frob_{v_1})=1$. (We make this last assumption in
order to simplify the deformation theory at $v_1$; in particular these
assumptions will imply that the local unrestricted lifting ring at
$v_1$ is smooth, and that all liftings are unramified.)
Let  $U=\prod_vU_v$ be a compact open
  subgroup of $G(\A_{F^+}^\infty)$ with $U_v$ a
compact open subgroup of $G(F^+_v)$ such that:
\begin{itemize}
\item $U_v= G(\bigO_{F^+_v})$ for all $v$ which split in $F$ other
  than $v_1$;
 
   \item  $U_{v_1}$ is  the preimage of the upper triangular  matrices under 
\[G(\cO_{F^+_{v_1}})\to G(k_{v_1}) \underset{\iota_{w_1}}\iso \GL_n(k_{v_1})\]
where $w_1$ is a place of $F$ over $v_1$;
  \item $U_v$ is a hyperspecial maximal compact subgroup of $G(F_v^+)$
    if $v$ is inert in $F$.
\end{itemize}
Then $U$ is sufficiently small (by the choice of $U_{v_1}$). Let
$T=S_p\cup \{v_1\}$. We let $\mathbb{T}^{T,\univ}$ be the commutative
$\bigO$-polynomial algebra generated by formal variables $T_w^{(j)}$
for all $1\le j\le n$, $w$ a place of $F$ lying over a place $v$ of
$F^+$ which splits in $F$ and is not contained in $T$. The algebra
$\mathbb{T}^{T,\univ}$ acts on $S_{\lambda,\tau}(U,\cO)$ via the
Hecke operators
  \[ T_{w}^{(j)}:=  \iota_{w}^{-1} \left[ \GL_n(\mc{O}_{F_w}) \left( \begin{matrix}
      \varpi_{w}1_j & 0 \cr 0 & 1_{n-j} \end{matrix} \right)
\GL_n(\mc{O}_{F_w}) \right] 
\] for $w\not \in T$ and $\varpi_w$ a uniformiser in
$\mc{O}_{F_w}$.
We denote by $\mathbb{T}^T_{\lambda,\tau}(U,\cO)$ the image of
$T^{T,\univ}$ in $\End_\cO(S_{\lambda,\tau}(U,\cO))$.

\begin{defn}\label{defn: max ideal or mod p Galois rep is
    automorphic}We say that a maximal ideal $\mf{m}$ of
  $\mathbb{T}^{T,\univ}$ with residue field of characteristic $p$ is
\emph{automorphic} if for some $(\lambda,\tau)$ as above we have
$S_{\lambda,\tau}(U,\cO)_{\mf{m}} \neq 0.$ We say that a
representation $\rhobar:G_{F^+}\to\cG_n(\Fpbar)$ is \emph{automorphic}
if there is an automorphic maximal ideal  $\mf{m}$ of
$\mathbb{T}^{T,\univ}$ such that  if $v\notin T$ is a place of $F^+$
which splits as $v=ww^c$ in $F$, then $\rhobar|_{G_F}(\Frob_w)$ has
characteristic polynomial equal to the image of
$X^n+\dots+(-1)^j(\mathbf{N}w)^{j(j-1)/2}T_w^{(j)}X^{n-j}+\dots(-1)^n(\mathbf{N}w)^{n(n-1)/2}T_w^{(n)}$. 
\end{defn}
In the following we will make a number of arguments that are vacuous
unless $S_{\lambda,\tau}(U,\cO)_{\mf{m}} \neq 0$ for the particular
$(\lambda,\tau)$ under consideration, but for technical reasons we do
not make this assumption.  Since $\rhobar$ is a suitable globalization
of $\rbar$ in the sense of Section~\ref{sec: patching for
  GLn}, it is automorphic by
assumption, and there is a maximal ideal $\mf{m}$ of
$\mathbb{T}^{T,\univ}$  associated to $\rhobar$ as in
Definition \ref{defn: max ideal or mod p Galois rep is
  automorphic}. Let $G_{F^+,T}:=\Gal(F(T)/F^+)$, where $F(T)$ is the
maximal extension of $F$ unramified outside of $T$ and infinity.
\begin{prop}\label{existence of repns} 
There is a unique continuous
lift \[\rho_\mf{m}:G_{F^+,T}\to\cG_n(\mathbb{T}^T_{\lambda,\tau}(U,\cO)_\mf{m})\]of
$\rhobar$, which satisfies
\begin{enumerate}
\item $\rho_{\mf{m}}^{-1}((\GL_n\times\GL_1)(\mathbb{T}^T_{\lambda,\tau}(U,\cO)_\mf{m}))=G_F$.
\item $\nu\circ\rho_\mf{m}=\varepsilon^{1-n}\delta_{F/F^+}^n$.
\item\label{item:char poly of universal Hecke deformation} $\rho_\mf{m}$ is unramified outside $T$. If $v\notin T$ splits
  as $ww^c$ in $F$ then $\rho_\mf{m}(\Frob_w)$ has characteristic
  polynomial \[X^n+\dots+(-1)^j(\mathbf{N}w)^{j(j-1)/2}T_w^{(j)}X^{n-j}+\dots(-1)^n(\mathbf{N}w)^{n(n-1)/2}T_w^{(n)}.\]
\item \label{item:local at l=p behaviour of universal Hecke
    deformation} For each place $v\in S_p$, and each $\cO$-algebra
  homomorphism
 $$x:\mathbb{T}^T_{\lambda,\tau}(U,\cO)_\mf{m}\to E',$$
  where $E'/E$ is a finite extension, the representation
  $x\circ\rho_\mf{m}|_{G_{F_\tv}}$ is potentially crystalline of Hodge
  type $\lambda_v$ and inertial type $\tau_v$.
\end{enumerate}
\end{prop}
\begin{proof} This may be proved in the same way as Proposition 3.4.4
  of \cite{cht}, making use of Corollaire 5.3 of~\cite{labesse},
  Theorem 1.1 of~\cite{blggtlocalglobalII}, and
 Theorem~\ref{thm: inertial local Langlands, N=0}. (Note that the
 purity assumption in Theorem~\ref{thm: inertial local Langlands, N=0}
 holds by Theorem~1.2 of~\cite{ana}, which builds on the results of
 \cite{shin} and \cite{chenevierharris}.)
\end{proof}
We note that our assumptions on $v_1$ imply that
$S_{\lambda,\tau}(U,\cO)_\m\otimes_{\Z_p}\Q_p$ is free over
$\mathbb{T}^T_{\lambda,\tau}(U,\cO)_\m[1/p]$ of rank $n!$ (where we
adopt the convention that this remark is true if both
$\mathbb{T}^T_{\lambda,\tau}(U,\cO)_\m[1/p]$ and
$S_{\lambda,\tau}(U,\cO)_\m\otimes_{\Z_p}\Q_p$ are zero). (This just
comes from the fact that the Iwahori invariants of an unramified
principal series representation of $\GL_n$ have dimension $n!$.)
\subsection{Deformations to $\cG_n$}
 Let $S$ be a set of places of $F^+$ which split in $F$, with
 $S_p\subseteq S$. As in
\cite{cht}, we will write $F(S)$ for the maximal extension of $F$
unramified outside $S$ and infinity, and from now on we will write
$G_{F^+,S}$ for $\Gal(F(S)/F^+)$. Regard $\rhobar$ as a representation
of $G_{F^+,S}$. We will freely make use of the terminology (of
liftings, framed liftings etc) of Section 2 of
\cite{cht}.

Choose a place $\tv_1$ of $F$ above $v_1$. Let $\tT$ denote the set of
places $\tv$, $v\in T$. For each $v\in T$, we let $R_\tv^\square$
denote the reduced and $p$-torsion free quotient of the universal
$\cO$-lifting ring of $\rhobar|_{G_{F_\tv}}$. For each $v\in S_p$, write
$R_\tv^{\square,\lambda_v,\tau_v}$ for
$R_{\rhobar|_{G_{F_\tv}},\lambda_v,\tau_v}^{\square}$.

Consider (in the terminology of \cite{cht}) the deformation problem
\[\cS:=(F/F^+,T,\tT,\cO,\rhobar,\varepsilon^{1-n}\delta_{F/F^+}^n,\{R_{\tv_1}^\square\}\cup\{R_\tv^{\square,\lambda_v,\tau_v}\}_{v\in S_p}).\]
There is a corresponding universal
deformation $\rho_\cS^{\univ}:G_{F^+,T}\to\cG_n(R_\cS^{\univ})$ of~$\rhobar$. In addition, there is a universal $T$-framed deformation ring
$R_\cS^{\square_T}$ in the sense of Definition 1.2.1 of~\cite{cht},
which parameterises deformations of $\rhobar$ of type $\cS$ together
with particular local liftings for each $\tv\in\tT$.
The lifting of Proposition \ref{existence of repns} and the universal
property of $\rho_\cS^{\univ}$ gives an
$\cO$-homomorphism \[R_\cS^{\univ}\onto
\mathbb{T}^T_{\lambda,\tau}(U,\cO)_\mf{m},\]which is surjective by
property (\ref{item:char poly of universal Hecke deformation}) of
$\rho_\mf{m}$.
\subsection{Patching}\label{subsec:patching} We now follow the proof
of Theorem 6.8 of \cite{jack}. We wish to consider auxiliary sets of
primes in order to apply the Taylor--Wiles--Kisin patching method. Since
$\rhobar(G_{F(\zeta_p)})$ is adequate by assumption, by Proposition 4.4
of \cite{jack} we see that we can (and do)
  choose an integer $q\ge[F^+:\Q]n(n-1)/2$ and for each $N\ge 1$ a tuple
  $(Q_N,\tQ_N,\{\psibar_\tv\}_{v\in Q_N})$ such that
\begin{itemize}
\item $Q_N$ is a finite set of finite places of $F^+$ of cardinality $q$ which is disjoint
  from $T$ and consists of places which split in $F$;
\item $\tQ_N$ consists of a single place $\tv$ of $F$ above each place
  $v$ of $Q_N$;
\item $\mathbf{N}v \equiv 1 \mod p^N$ for $v \in Q_N$;
\item for each $v\in Q_N$,
  $\rhobar|_{G_{F_\tv}}\cong\sbar_\tv\oplus\psibar_\tv$ where $\psibar_\tv$ is
  an eigenspace of Frobenius corresponding to an eigenvalue
  $\alpha_v$, on which Frobenius acts semisimply. Write $d_N^v=\dim\psibar_\tv$.
\end{itemize}
For each $v\in Q_N$, let $R^{\psibar_\tv}_{\tv}$ denote the
quotient of $R^\square_{\tv}$ corresponding to lifts
$r:G_{F_\tv}\to\GL_n(A)$ which are
$\ker(\GL_n(A)\to\GL_n(k))$-conjugate to a lift of the form
$s\oplus\psi$, where $s$ is an unramified lift of $\sbar_{\tv}$ and
$\psi$ is a lift of $\psibar_\tv$ for which the image of inertia is
contained in the scalar matrices. We let $\cS_{Q_N}$ denote the
deformation problem 
\begin{multline*}
\cS_{Q_N} := (F/F^+,\, T\cup Q_N,\, \tT\cup
\tQ_N,\, \cO,\, \rhobar,\, \varepsilon^{1-n}\delta_{F/F^+}^n, \\
 \{R_{\tv_1}^\square\}
 \cup\{R_\tv^{\square,\lambda_v,\tau_v}\}_{v\in S_p}
 \cup\{R_{\tv}^{\psibar_\tv}\}_{v\in Q_N}).
\end{multline*}
We let $R_{\cS_{Q_N}}^{\univ}$
denote the corresponding universal deformation ring, and we let
$R^{\square_T}_{\cS_{Q_N}}$ denote the corresponding universal $T$-framed deformation
ring.
We define \[ R^{\loc} := \left(\widehat{\otimes}_{v \in
  S_p}R_\tv^{\square,\lambda_v,\tau_v}\right)\widehat{\otimes}{R}^{\square}_{\tv_1} \]
where all completed tensor products are taken over $\mc{O}$. By Proposition 4.4 of \cite{jack}, we may also assume that 
\begin{itemize}
\item the ring $R^{\square_T}_{\mc{S}_{Q_N}}$ can be topologically
  generated over $R^{\loc}$ by $q-[F^+:\Q]n(n-1)/2$ elements.
\end{itemize}
Let $U_0(Q_N)=\prod_v U_0(Q_N)_v$ and $U_1(Q_N)=\prod_v U_1(Q_N)_v$ be
the compact open subgroups of $G(\A_{F^+}^\infty)$ defined by (for
$i=0$, $1$) $U_i(Q_N)_v=U_v$ if $v\notin Q_N$, and
$U_0(Q_N)_v=\iota_{\tv}^{-1}\p_N^\tv$,
$U_1(Q_N)_v=\iota_{\tv}^{-1}\p_{N,1}^\tv$ if $v\in Q_N$, where
$\p_N^\tv$ and $\p_{N,1}^\tv$ are the parahoric subgroups defined in
\cite{jack}, corresponding to the partition $n=(n-d_N^\tv)+d_N^\tv$.
We have natural maps \[ \bb{T}_{\lambda,\tau}^{T\cup
  Q_N}(U_{1}(Q_N),\mc{O}) \onto \bb{T}_{\lambda,\tau}^{T\cup
  Q_N}(U_{0}(Q_N),\mc{O}) \onto \bb{T}_{\lambda,\tau}^{T\cup
  Q_N}(U,\mc{O}) \into \bb{T}_{\lambda,\tau}^T(U,\mc{O}).\] Thus
$\mf{m}$ determines maximal ideals of the first three algebras in this
sequence which we denote by $\mf{m}_{Q_N}$ for the first two and
$\mf{m}$ for the third. Note also that $\bb{T}_{\lambda,\tau}^{T \cup
  Q_N}(U,\mc{O})_{\mf{m}} =
\bb{T}_{\lambda,\tau}^T(U,\mc{O})_{\mf{m}}$ by the proof of Corollary
3.4.5 of \cite{cht}.
For each $v \in Q_N$ choose an element
$\phi_{\wt{v}} \in G_{F_{\wt{v}}}$ lifting the geometric Frobenius
element of $G_{F_{\tv}}/I_{F_{\tv}}$ and
let $\varpi_{\wt{v}} \in \mc{O}_{F_{\wt{v}}}$ be the uniformiser with
$\Art_{F_{\wt{v}}}\varpi_{\wt{v}}=\phi_{\wt{v}}|_{F_{\wt{v}}^{\ab}}$. As
in Proposition 5.9 of \cite{jack} there are commuting projection
operators $\pr_{\varpi_{\tv}}\in\End_{\cO}(S_{\lambda,\tau}(U_i(Q_N),\cO)_{\m_{Q_N}})$. Write $M=S_{\lambda,\tau}(U,\cO)_\m$, and for $i=0$, $1$ we
write \[M_{i,Q_N}=\left(\prod_{v\in
  Q_N}\pr_{\varpi_\tv}\right)S_{\lambda,\tau}(U_i(Q_N),\cO)_{\m_{Q_N}}.\]Let
$\bb{T}_{i,Q_N}$ denote the image of $\bb{T}_{\lambda,\tau}^{T\cup
  Q_N}(U_i(Q_N),\cO)$ in $\End_\cO(M_{i,Q_N})$,
let $\Delta_{Q_N}=U_0(Q_N)/U_1(Q_N)$, and let $\mf{a}_{Q_N}$
denote the kernel of the
augmentation map $$\mc{O}[\Delta_{Q_N}] \rightarrow \mc{O}.$$
For
$i=0,1$ and $\alpha \in F_{\wt{v}}$ of non-negative valuation, we have
the Hecke operator \[V_\alpha=\iota^{-1}_\tv\left(\left[\p_{N,1}^\tv
    \begin{pmatrix}
      1_{n-d_N^\tv}&0\\0&A_\alpha
    \end{pmatrix}\p_{N,1}^\tv\right]\right),\]where
$A_\alpha=\diag(\alpha,1,\dots,1)$. Exactly as in the proof of Theorem
6.8 of \cite{jack}, we have:
\begin{enumerate}
\item The  map
\[\bigl(\prod_{v\in
  Q_N}\pr_{\varpi_\tv}\bigr):M\to M_{i,Q_N}  \]
is an isomorphism.
\item $M_{1,Q_N}$ is free over $\mc{O}[\Delta_{Q_N}]$ with
\[ M_{1,Q_N}/\mf{a}_{Q_N}\isoto M_{0,Q_N}. \]
\item For each $v \in Q_N$, there is a character with
  open kernel $V_{\wt{v}} : F_{\wt{v}}^{\times} \rightarrow
  \bb{T}_{1,Q_N}^{\times}$ so that
  \begin{enumerate}
  \item for each $\alpha \in F_{\wt{v}}$ of non-negative valuation,
    $V_{\alpha}= V_{\wt{v}}(\alpha)$ on $M_{1,Q_N}$;
  \item $(\rho_{\mf{m}_{Q_N}}\otimes_{\bb{T}_{\lambda,\tau}^{T\cup
  Q_N}(U_i(Q_N),\cO)_{\m_{Q_N}}} \bb{T}_{1,Q_N})|_{W_{F_{\wt{v}}}}
    \cong s\oplus\psi$ with $s$ an unramified lift of $\sbar_\tv$ and
    $\psi$ a lift of $\psibar_{\wt{v}}$ with $I_{F_\tv}$ acting on
    $\psi$ via the scalar $(V_{\wt{v}}\circ \Art_{F_{\wt{v}}}^{-1})$.
  \end{enumerate}
\end{enumerate}
 The above shows, in particular, that 
the lift $\rho_{\mf{m}_{Q_N}} \otimes \bb{T}_{1,Q_N}$ of $\rhobar$ is of
type $\mc{S}_{Q_N}$ and gives rise to a surjection
$R^{\univ}_{\mc{S}_{Q_N}} \onto \bb{T}_{1,Q_N}$. 
Thinking of $\Delta_{Q_N}$ as the product of the inertia subgroups in
the maximal abelian $p$-power order quotient of $\prod_{v \in Q_N}
G_{F_{\wt{v}}}$, we obtain a homomorphism $\Delta_{Q_N} \rightarrow
(R^{\univ}_{\mc{S}_{Q_N}})^{\times}$ by considering $\psi$ as above in
some basis. We thus have homomorphisms
$$\mc{O}[\Delta_{Q_N}] \rightarrow 
R^{\univ}_{\mc{S}_{Q_N}} \rightarrow R^{\square_T}_{\mc{S}_{Q_N}}$$
and natural isomorphisms $R^{\univ}_{\mc{S}_{Q_N}}/\mf{a}_{Q_N} \cong
R^{\univ}_{\mc{S}}$ and $R^{\square_T}_{\mc{S}_{Q_N}}/\mf{a}_{Q_N}
\cong R^{\square_T}_{\mc{S}}$, and the surjection
$R^{\univ}_{\mc{S}_{Q_N}} \onto \bb{T}_{1,Q_N}$ is a homomorphism of
$\cO[\Delta_{Q_N}]$-algebras.

Fix a filtration by $\F$-subspaces
\[0=L_0\subset L_1\subset\dots\subset L_s=L_{\lambda,\tau}/\pi L_{\lambda,\tau}\]
such that each
$L_i$ is $G(\cO_{F^+,p})$-stable, and for each $i=0,1,\dots,s-1$, the quotient
$\sigma_i:=L_{i+1}/L_i$ is absolutely irreducible. This induces a
filtration \[0=M^{0}\subset M^{1}\subset\dots\subset M^{s}=M/\pi M.\]

We may now patch just as in the proof of 
\cite[Thm.~6.8]{jack}, keeping track of filtrations as in
\cite[\S~(2.2.9)]{kisinfmc}. (See also section 4.3 of~\cite{geekisin}
in the case $n=2$.) More precisely, we patch together the
$R^{\loc}[\Delta_{Q_N}]$-modules
$M_{1,Q_N}\otimes_{R_{\cS_{Q_N}}^\univ}R_{\cS_{Q_N}}^{\square_T}$ via
Lemma 6.10 of~\cite{jack}. Having made this construction, if we set \[\Rinfty:=\left(\widehat{\otimes}_{v \in
  S_p,\cO}R_\tv^{\square}\right)\widehat{\otimes}_\cO{R}^{\square}_{\tv_1}[[x_1,\dots,x_{q-[F^+:\Q]n(n-1)/2}]],\]
\begin{multline*}
\Rbarinfty :=
\left(\widehat{\otimes}_{v \in
  S_p,\cO}R_\tv^{\square,\lambda_v,\tau_v}\right)\widehat{\otimes}_\cO{R}^{\square}_{\tv_1}
[[x_1,\dots,x_{q-[F^+:\Q]n(n-1)/2}]]
\\
= R^{\loc}[[x_1,\dots,x_{q - [F^+:\Q]n(n-1)/2}]]
,
\end{multline*}
and
$$S_\infty := \cO[[z_1,\dots,z_{n^2\#T},y_1,\dots,y_q]],$$ for formal
variables $x_1,\dots,x_{q-[F^+:\Q]n(n-1)/2}$, $y_1,\dots,y_q$ and
$z_1,\dots,z_{n^2\#T}$, and if we let $\mf{a}$ denote the kernel of the
augmentation map $S_\infty\to\cO$, then we see that there exists:
\begin{itemize}
\item An $S_{\infty}$-module $M_{\infty}$ which is simultaneously an $\Rbarinfty$-module 
such that the image of $\Rbarinfty$ in $\End(M_{\infty})$
is an $S_{\infty}$-algebra. 
\item A filtration by $\Rbarinfty$-modules 
\[0=M_{\infty}^{0}\subset M_{\infty}^{1}\subset\dots\subset
M_{\infty}^{s}=M_{\infty}/\pi M_\infty\] 
whose graded pieces are finite free $S_{\infty}$-modules. 
\item A surjection of $R^{\loc}$-algebras $\Rbarinfty \rightarrow R^{\univ}_{\mc{S}}.$
\item An isomorphism of $\Rbarinfty$-modules
$$M_{\infty}/\mf{a}M_{\infty} \iso M,$$
which identifies 
$M^i_{\infty}/\mf{a}M^i_{\infty}$ with $M^i,$
and such that the induced $\Rbarinfty$-module structure on $M$ 
coincides with the $R^{\univ}_{\mc{S}}$-module structure on $M$,
via the surjection $\Rbarinfty \to R^{\univ}_{\mc{S}}$ mentioned
in the previous point.
\end{itemize}
We furthermore
claim that we can make the above construction so that, for each value of $i=1,2,\dots s,$ 
the $(\Rinfty,S_{\infty})$-bimodule structure on $M^i_{\infty}/M^{i-1}_\infty$
(arising from its $(\Rbarinfty,S_{\infty})$-bimodule structure, and the
surjection $\Rinfty\to \Rbarinfty$) and the isomorphism 
$M^i_{\infty}/(\mf{a}M^i_{\infty},M^{i-1}_\infty) \iso M^i/M^{i-1}$ depends only on ($U,\mf{m}$ and) the isomorphism class 
of $L_i/L_{i-1}$ as a $G(\cO_{F^+,p})$-representation, but not on $(\lambda,\tau).$ For any finite 
collection of pairs $(\lambda,\tau)$ this follows by the same finiteness argument used during patching. 
Since the set of $(\lambda,\tau)$ is countable, the claim follows from
a diagonalization argument.

It will be convenient in the following to refer to a tuple
$(a_v)_{\tv\in\tS_p}$ as a Serre weight, where each $a_v$ is a Serre
weight for $\GL_n(k_\tv)$, where $k_\tv$ is the residue field of
$F_\tv$. We write $F_a=\otimes_{\tv\in\tS_p}F_{a_v}$, a representation
of $G(\cO_{F^+,p})$. For $a$ a Serre weight, we denote by
$M^a_{\infty}$ the $\Rinfty/\pi$-module constructed
above (using the construction for all $(\lambda,\tau)$) when
$L_i/L_{i-1} \iso F_a,$ and we set
\[ \mu_a'(\rhobar) = (n!)^{-1}e(M^a_\infty,\Rinfty/\pi)\]
and \[ \cycle_a'(\rhobar) = (n!)^{-1}\cycle(M^a_\infty).\](Note that \emph{a
  priori} this is only a cycle with $\Q$-coefficients rather than
$\Z$-coefficients, but in the situations in which our theorems apply,
it will follow that it is in fact a cycle with $\Z$-coefficients.)

For each $v|p$ we
write \[(L_{\lambda_v,\tau_v}\otimes_\cO\F)^\semis\isoto\sum_{a_v}F_{a_v}^{n_{a_v}},\]
so
that \[(L_{\lambda,\tau}\otimes_\cO\F)^\semis\isoto\sum_{a}F_{a}^{n_{a}},\]where $n_a=\prod_vn_{a_v}$.

The following technical, but crucial, lemma is a slight refinement (and generalisation to the
$n$-dimensional setting) of Lemma~4.3.8 of \cite{geekisin}.
\begin{lem}\label{key patching lemma} For each $a,$
  $\mu'_a(\rhobar)$ is a non-negative integer. Moreover, for any fixed $(\lambda,\tau)$ the following conditions are equivalent: 

(1) The support of $M\otimes_{\Z_p}\Q_p$
meets every irreducible component of $\Spec R^{\loc}[1/p].$

(2) $M_{\infty}\otimes_{\Z_p}\Q_p$ is a faithfully flat $\Rbarinfty[1/p]$-module 
which is everywhere locally free of rank $n!.$

(3) $R^{\univ}_{\mc{S}}$ is a finite $\cO$-algebra and
$M\otimes_{\Z_p}\Q_p$ is a faithful $R^{\univ}_{\mc{S}}[1/p]$-module.  

(4) $e(\Rbarinfty/\pi) =  (n!)^{-1}\sum_a
n_ae(M^{a}_{\infty},\Rinfty/\pi) = \sum_a n_a\mu_{a}'(\rhobar).$

(5) $\cycle(\Rbarinfty/\pi) =  (n!)^{-1}\sum_a n_a\cycle(M^{a}_{\infty}) = \sum_a n_a\cycle_{a}'(\rhobar).$
\end{lem}
\begin{proof} By Proposition \ref{prop: existence of local deformation
    rings} and our assumptions on the primes in $T,$ $\Rbarinfty[1/p]$
  is formally smooth, and equidimensional of dimension $q+n^2\#T = \dim
  S_{\infty}[1/p].$   Furthermore, the morphism
$\Spec \Rbarinfty[1/p] \to \Spec R^{\loc}[1/p]$, corresponding to
the $R^{\loc}[1/p]$-algebra structure on $\Rbarinfty[1/p]$, induces
a bijection on irreducible components; below we denote this bijection
by $Z \mapsto Z'$.

Since $M_{\infty}$ is finite free over
  $S_{\infty}$, and the image of $\Rbarinfty$ in $\End(M_{\infty})$ is
  an $S_{\infty}$-algebra, $M_{\infty}\otimes_{\Z_p}\Q_p$ has depth at
  least $q+n^2\#T$ at every maximal ideal of $\Rbarinfty[1/p]$ in its
  support; so by the Auslander--Buchsbaum formula, it has depth
  exactly $q+n^2\#T$ at every maximal ideal of $\Rbarinfty[1/p]$ in
  its support, and $M_{\infty}\otimes_{\Z_p}\Q_p$ has projective
  dimension $0$ over $\Rbarinfty[1/p]$. Since
  $M_{\infty}\otimes_{\Z_p}\Q_p$ is finite over $S_{\infty}[1/p]$, we
  see that it is finite flat over $\Rbarinfty[1/p],$
and so its support is a union of irreducible components of
$\Spec \Rbarinfty[1/p]$.

If $Z \subset \Spec \Rbarinfty[1/p]$ is an irreducible component in the support of $M_{\infty}\otimes_{\Z_p}\Q_p,$ then $Z$ is 
finite over $\Spec S_{\infty}[1/p]$ and of the same dimension, because
$M_{\infty}\otimes_{\Z_p}\Q_p$ is finite free over $S_{\infty}[1/p]$. Hence the map $Z \rightarrow \Spec S_{\infty}[1/p]$ 
is surjective.  In particular, the fibre of $Z$ over
the closed point $\mf{a}$ of $\Spec S_{\infty}[1/p]$ is non-zero,
and hence gives a point in the support of
$M_{\infty}/\mf{a}M_{\infty} = M$
lying in
the component $Z'$ of $\Spec R^{\loc}[1/p]$ corresponding to $Z$.
As $M$ has rank $n!$ over any point of $R^{\univ}_{\mc{S}}[1/p]$ in its support,
it follows that $M_{\infty}\otimes_{\Z_p}\Q_p$ is in fact
locally free of rank $n!$ over all of $Z.$ 

Thus we see that $M_{\infty}\otimes_{\Z_p}\Q_p$ is faithful over
$\Rbarinfty[1/p]$, or equivalently, has support consisting of
the union of all the irreducible components of $\Rbarinfty[1/p]$,
if and only if the support of $M$ meets each irreducible component
of $\Spec R^{\loc}[1/p]$, and that, in this case, it is furthermore
locally free of rank $n!$.
This shows that (1) and (2) are equivalent.

Using Proposition 1.3.4 of \cite{kisinfmc} we also see that each $\mu'_a(\rhobar)$ is a non-negative integer, and that 
\[e(\Rbarinfty/\pi) \geq (n!)^{-1} e(M_{\infty}/\pi M_{\infty},\Rbarinfty/\pi) 
=  (n!)^{-1}\sum_an_ae(M^{a}_{\infty},\Rinfty/\pi)\]
 with equality if and only if $M_{\infty}$ is a faithful
 $\Rbarinfty$-module (in which case, as already observed,
it is necessarily locally free of rank $n!$).
Thus (2) and (4) are equivalent.

If $M_{\infty}$ is a faithful $\Rbarinfty$-module, then $\Rbarinfty$,
being a subring of the ring of $S_\infty$-endomorphisms of the finite
$S_\infty$-module $M_\infty$, is finite over $S_{\infty},$ 
and so $R^{\univ}_{\mc{S}},$ which is a quotient of $\Rbarinfty/\mf{a},$ is a finite $\cO$-module. This shows that (2) implies (3).
Now assume (3), so that $R^{\univ}_{\mc{S}}$ is a finite $\cO$-algebra; we
see by Proposition 1.5.1 of \cite{BLGGT} that the image of 
$$ \Spec R^{\univ}_{\mc{S}} \rightarrow \Spec R^{\loc}$$ 
meets every component of $\Spec R^{\loc}[1/p].$ Hence, since
$M\otimes_{\Z_p}\Q_p$ is a faithful $R^{\univ}_{\mc{S}}$-module by assumption, 
we see that~(1) holds.
Thus (1)--(4) are all equivalent.

It remains to show that (5) is equivalent to the other
conditions. From Lemma \ref{lem:multiplicities from cycles}, we see
that (5) implies (4). Now assume that (2) holds. Then $M_\infty$ is
$\pi$-torsion free and generically free of rank $n!$ over each
component of $\Spec \Rbarinfty$, so by Proposition~\ref{prop:cutting out}
we have \[\cycle(\Rbarinfty/\pi)=(n!)^{-1}\cycle(M_\infty/\pi
M_\infty).\] Identifying each of the modules $M_{\infty}/\pi
M_{\infty}$, $M^i_{\infty}/M^{i-1}_\infty$, etc., with the
corresponding sheaves on $\Spec \Rinfty/\pi$ that they give rise
to, we see from Lemma \ref{lem:cycles in exact sequences} that we
have \begin{align*}\cycle(M_{\infty}/\pi M_\infty)
  &=\sum_i\cycle(M^i_\infty/M^{i-1}_\infty)
  \\&=\sum_an_a\cycle(M^{a}_{\infty}).\end{align*} The result
follows.\end{proof}
For each Serre weight
  $a\in(\Z^n_+)^{\Hom(k,\F)}$, let $\lambda_a$ be a fixed lift of
  $a$. As in Remark \ref{rem: generalised BM conjecture remarks}(1),
  we may inductively define unique integers $\mu_a(\rbar)$ such that
  for any Serre weight $b$, if we
  write \[(L_{\lambda_b,\triv}\otimes_\cO\F)^\semis\isoto\oplus_a
  F_a^{m_{a,b}},\] then \[e(R^{\square}_{\rbar,\lambda_b,\triv}/\pi)=\sum_a
  m_{a,b}\mu_a(\rbar).\] (In Remark \ref{rem: generalised BM
    conjecture remarks}(1) we were assuming the Breuil--M\'ezard
  conjecture, but the inductive process defining the $\mu_a(\rbar)$
  does not rely on any cases of the conjecture.) Similarly, we may inductively define unique
  cycles $\cC_a$ such that \[\cycle(R^{\square}_{\rbar,\lambda_b,\triv}/\pi)=\sum_a
  m_{a,b}\cC_a.\] 

We now prove the main result of this section, a conditional
equivalence between the geometric formulation of the Breuil--M\'ezard
conjecture, the numerical formulation, and a global statement. For the
convenience of the reader, we recall the various assumptions of this
section in the statement of the theorem.
\begin{thm}
  \label{thm: purely local statement assuming that for Serre weights,
    and in addition the statement of the support for lambda and tau}
  Let $p>2$ be prime, let $K/\Qp$ be a finite extension, and let
  $\rbar:G_K\to\GL_n(\F)$ be a continuous representation, with $\F$ a
  finite extension of $\Fp$. Assume that there is a suitable
  globalization $\rhobar$ of $\rbar$ as
  in Section~{\em \ref{subsec:basics}} {\em (}for example, by Corollary {\em \ref{cor: the final
    local-to-global result}} this is guaranteed if $p\nmid n$ and
  Conjecture {\em \ref{conj: existence of local crystalline lifts}} holds
  for $\rbar${\em )}.

  Suppose that the equivalent conditions of Lemma~{\em \ref{key patching
    lemma}} hold whenever each $\lambda_v=\lambda_{a_v}$ for some Serre
  weight $a_v$ and each $\tau_v=\triv$.

Then, if
$\lambda\in(\Z^n_+)^{\Hom_\Qp(K,E)}$ and $\tau$ is an inertial type 
for $G_K$, and if we write 
\[(L_{\lambda,\tau}\otimes_\cO\F)^\semis\isoto\oplus_a F_a^{n_a},\]
the following conditions are equivalent:
\begin{enumerate}
\item The equivalent conditions of Lemma~{\em \ref{key patching lemma}} hold
  when $\lambda_v=\lambda$ and $\tau_v=\tau$ for each $v|p$.
\item $e(R^{\square}_{\rbar,\lambda,\tau}/\pi)=\sum_a
  n_a\mu_a(\rbar)$.
\item $\cycle(R^{\square}_{\rbar,\lambda,\tau}/\pi)=\sum_a
  n_a\cC_a$.
\end{enumerate}

\end{thm}
\begin{proof} Consider conditions (4) and (5)
  of Lemma \ref{key patching lemma} in the case that each
  $\lambda_v=\lambda_{b_v}$ for some Serre weight $b_v$ and each
  $\tau_v=\triv$. Note that in this case we have
  $(L_{\lambda,\tau}\otimes_\cO\F)^\semis\isoto\oplus_aF_a^{m_{a,b}}$
  where $m_{a,b}=\prod_vm_{a_v,b_v}$. Throughout this proof we will
  write $\rhobar_v$ for $\rhobar|_{G_{F_\tv}}$ (which is isomorphic to
  $\rbar$).

  As remarked above, the conditions on $v_1$ ensure that $R_{\tv_1}^\square$ is
  formally smooth over $\cO$, say $R_{\tv_1}^\square\cong\cO[[t_1,\dots,t_{n^2}]]$. Our
  assumption that the equivalent conditions of Lemma~\ref{key patching
    lemma} hold implies that \begin{align*}\sum_am_{a,b}\mu'_{a}(\rhobar)&=e(\Rbarinfty/\pi)\\&=\prod_{v|p}e(R^{\square}_{\rhobar_v,\lambda_{b_v},\triv}/\pi)\\&=\prod_{v|p}\sum_{a_v}
    m_{a_v,b_v}\mu_{a_v}(\rbar)\\&=\sum_am_{a,b}\prod_{v|p}\mu_{a_v}(\rbar),\end{align*}
  where the second equality holds by Proposition 1.3.8 of
  \cite{kisinfmc}. Similarly \begin{align*}\sum_am_{a,b}\cycle'_{a}(\rhobar)&=\cycle(\Rbarinfty/\pi)\\&=\prod_{v|p}\cycle(R^{\square}_{\rhobar_v,\lambda_{b_v},\triv}/\pi)\times\cycle
    (\F[[[x_1,\dots,x_{q-[F^+:\Q]n(n-1)/2},t_1,\dots,t_{n^2} ]])\\&=\prod_{v|p}\sum_{a_v}
    m_{a_v,b_v}\cC_{a_v}\times\cycle
    (\F[[[x_1,\dots,x_{q-[F^+:\Q]n(n-1)/2},t_1,\dots,t_{n^2} ]])\\&=\sum_am_{a,b}\prod_{v|p}\cC_{a_v}\times\cycle
    (\F[[[x_1,\dots,x_{q-[F^+:\Q]n(n-1)/2},t_1,\dots,t_{n^2} ]]),\end{align*}where we
  have made use of Lemma \ref{lem:product cycles}. If we define a partial order on the tuples $(a_{v})_{v|p}$ of Serre
  weights by writing $(a_{v})_{v|p}\ge (b_{v})_{v|p}$ to mean that
  each $a_v\ge b_v$, then an easy induction using Lemma \ref{lem:
    decomposition of L_lambda mod p} shows that we must in fact have
  that for each Serre weight
  $a$, \[\mu'_a(\rhobar)=\prod_{v|p}\mu_{a_v}(\rbar)\]and \[\cycle'_a(\rhobar)=\prod_{v|p}\cC_{a_v}\times\cycle
    (\F[[[x_1,\dots,x_{q-[F^+:\Q]n(n-1)/2},t_1,\dots,t_{n^2} ]]).\]
  
We now consider the conditions of Lemma \ref{key patching lemma} in
the case that each $\lambda_v=\lambda$ and each $\tau_v=\tau$. By
definition (and Lemma \ref{lem:product cycles}), we
have \[\cycle(\Rbarinfty/\pi)=\prod_{v|p}\cycle(R^{\square}_{\rhobar_v,\lambda,\tau}/\pi)\times\cycle
    (\F[[[x_1,\dots,x_{q-[F^+:\Q]n(n-1)/2},t_1,\dots,t_{n^2} ]]).\]
Using the relation above, we
have \begin{align*}\sum_an_a\cycle_{a}'(\rhobar)&=\sum_a\prod_{v|p}n_{a_v}\cC_{a_v}\times\cycle
    (\F[[[x_1,\dots,x_{q-[F^+:\Q]n(n-1)/2},t_1,\dots,t_{n^2} ]])\\&=\prod_{v|p}\sum_{a_v}n_{a_v}\cC_{a_v}\times\cycle
    (\F[[[x_1,\dots,x_{q-[F^+:\Q]n(n-1)/2},t_1,\dots,t_{n^2}
    ]]). \end{align*}Now, condition (5) of Lemma~\ref{key patching
    lemma} states
  that  \[\cycle(\Rbarinfty/\pi)=\sum_an_a\cycle'_{a}(\rhobar),\] which  is equivalent
to \begin{multline*}\prod_{v|p}\cycle(R^{\square}_{\rhobar_v,\lambda,\tau}/\pi)\times\cycle
    (\F[[[x_1,\dots,x_{q-[F^+:\Q]n(n-1)/2},t_1,\dots,t_{n^2} ]])\\=\prod_{v|p}\sum_{a_v}n_{a_v}\cC_{a_v}\times\cycle
    (\F[[[x_1,\dots,x_{q-[F^+:\Q]n(n-1)/2},t_1,\dots,t_{n^2} ]])\end{multline*}
and thus
to \[\cycle(R^{\square}_{\rbar,\lambda,\tau}/\pi)=\sum_an_a\cC_a,\]giving
the equivalence of (1) and (3). The equivalence of (1) and (2) may be proved
by a formally identical argument.  
\end{proof}
\begin{remark}
  {\em We regard the assumption of Theorem \ref{thm: purely local
      statement assuming that for Serre weights, and in addition the
      statement of the support for lambda and tau} (that the
    conditions of Lemma \ref{key patching lemma} hold for all Serre
    weights) as a strong form of the weight part of Serre's
    conjecture, saying that each set of components of the local
    crystalline deformation rings in low weight is realised by a
    global automorphic Galois representation. Under this assumption, Theorem \ref{thm: purely local statement assuming that
      for Serre weights, and in addition the statement of the support
      for lambda and tau} may be regarded as saying that instances of
    the Breuil--M\'ezard conjecture (and its geometric refinement) are
    equivalent to certain $R=\T$ theorems.}
\end{remark}
One case in which Theorem~\ref{thm: purely local statement assuming that for Serre weights,
    and in addition the statement of the support for lambda and tau}
  has unconditional consequences is the case that $n=2$ and
  $\lambda=0$, where the automorphy lifting theorems of \cite{kis04}
  and \cite{MR2280776} can be applied. We now give such an
  application, where we make use of the results of~\cite{geekisin} to
  give a particularly clean statement. Note that when $n=2$, the
  inductive definition of the cycles $\cC_a$ above is trivial, and we
  simply set $\cC_a:=\cycle(R^{\square}_{\rbar,\lambda_a,\triv}/\pi)$.

Suppose that $K/\Qp$ is
  absolutely unramified, and that $n=2$; then we say that a Serre
  weight $a$ is a \emph{predicted Serre weight for} $\rbar$ if it is one of
  the weights predicted in~\cite{bdj}. We say that $a$ is
  \emph{regular} if we have $0\le a_{\sigma,1}-a_{\sigma,2}<p-1$ for
  each $\sigma\in\Hom(k,\F)$.

\begin{thm}
  \label{thm: purely local statement for unramified regular pot-BT}
  Let $p>2$ be prime, let $K/\Qp$ be an absolutely unramified finite extension, and let
  $\rbar:G_K\to\GL_2(\F)$ be a continuous representation, with $\F$ a
  finite extension of $\Fp$. Suppose that every predicted Serre weight for
  $\rbar$ is regular.
  Then, for any inertial type $\tau$ for $G_K$, if we write
$(L_{0,\tau}\otimes_\cO\F)^\semis\isoto\oplus_a F_a^{n_a},$
we have $\cycle(R^{\square}_{\rbar,0,\tau}/\pi)=\sum_a n_a\cC_a$.
\end{thm}
\begin{proof}
  This will follow from Theorem~\ref{thm: purely local statement assuming that for Serre weights,
    and in addition the statement of the support for lambda and tau}
  once we have checked that all of the hypotheses hold. Since $n=2$ and $p>2$,
  we have $p\nmid n$ and thus a suitable globalization will exist
  provided that Conjecture~\ref{conj: existence of local crystalline
    lifts} holds for $\rbar$; but this follows from the proof of
  Theorem A.1.2 of~\cite{geekisin} (which shows that $\rbar$ has a
  potentially Barsotti--Tate lift) and Lemma 4.4.1 of \emph{op.cit}.\ 
  (which shows that any potentially Barsotti--Tate representation is
  potentially diagonalizable).

  It remains to check that the equivalent conditions of Lemma 5.5.1 hold
  whenever $\lambda_v=0$ and $\tau_v=\tau$ (for an arbitrary $\tau$)
  for all $v|p$, and whenever $\tau_v$ is trivial for all $v$ and each
  $\lambda_v$ corresponds to a Serre weight. In the first
  (respectively second) case, this follows from Lemma 4.4.1
  (respectively Lemma 4.4.2) of \cite{geekisin}, exactly as in the
  proof of Corollary 4.4.3 of \cite{geekisin}.
\end{proof}
\begin{remark}
  {\em In particular, Theorem~\ref{thm: purely local statement for unramified regular pot-BT}
  confirms Conjecture 1.4 of~\cite{breuilmezardIII} in the case that
  each $k_\tau=2$, and extends Th\'eor\`eme 1.5 of \emph{ibid}.\ to
  arbitrary types.}
\end{remark}
\appendix
\section{Realising local representations globally}\label{sec:local to
  global}
In this appendix we realise local representations globally, using
potential automorphy theorems. In the case $n=2$, analogous results
were proved by similar techniques in Appendix A of \cite{geekisin}. We
would like to thank Robert Guralnick and Florian Herzig for the proof
of the following fact.
\begin{alemma}\label{lem: adequacy for p^n sufficiently large}Suppose
  that $p>2$ and $p\nmid n$. Let
  $(\GL_n.2)(\F_{p^m})$ denote the  subgroup of $\GL_{2n}(\F_{p^m})$
  generated by block diagonal matrices of the form $(g,{}^{-1}g^t)$
  and a matrix $J$ with $J(g,{}^{-1}g^t)J^{-1}=({}^{-1}g^t,g)$. Then
  for $m$ sufficiently large, both
  $(\GL_n.2)(\F_{p^m})\subset\GL_{2n}(\Fpbar)$ and
  $\GL_n(\F_{p^m})\subset\GL_n(\Fpbar)$ are adequate.
  \end{alemma}
  \begin{proof}
    In the case of $\GL_n(\F_{p^m})$, this is immediate from Theorem 1.2 and
    Lemma 1.4
    of~\cite{1208.4128}. For $(\GL_n.2)(\F_{p^m})$, since we are
    assuming that $(p,2n)=1$, it follows from Theorem 1.5
    of~\cite{guralnickadequateII} (which shows the corresponding
    result for the underlying algebraic group $\GL_n.2$)
    and the proof of Theorem 1.2 of~\cite{1208.4128} that the result
    will be true provided that $\Ext^1_{(\GL_n.2)(\F_{p^m})}(V,V)=0$
    for $m$ sufficiently large, where $V=\F_{p^m}^{2n}$ is the space
    on which $(\GL_n.2)(\F_{p^m})$ acts. 

The result then follows easily from
    inflation-restriction, the corresponding result for $\GL_n(\F_{p^m})$ (which
    is proved in the course of the proof of Theorem 1.2
    of~\cite{1208.4128}), and the fact that the standard representation of
    $\GL_n(\F_{p^m})$ has a different central character to its dual
    provided that $p^m>3$.
  \end{proof}
\begin{aprop}\label{prop: local realisation without modularity}
  Let $K/\Qp$ be a finite extension, and let
  $\rbar:G_K\to\GL_n(\Fpbar)$ be a continuous representation, with
  $p\nmid n$. Then
  there is a CM field $L$ with maximal totally real subfield $L^+$, and a continuous irreducible
  representation $\rhobar:G_{L^+}\to\cG_n(\Fpbar)$, such that
  \begin{itemize}
  \item each place $v|p$ of $L^+$ splits in $L$;
  \item for each place $v|p$ of $L^+$, $L^+_v\cong K$ and there is a
    place $\tv$ of $L$ lying over $v$ such that $\rhobar|_{G_{L_\tv}}\cong\rbar$;
 \item $\nu\circ\rhobar=\varepsilonbar^{1-n}\delta_{L/L^+}^n$, where
     $\delta_{L/L^+}$ is the quadratic character corresponding to $L/L^+$;
  \item $\rhobar^{-1}(\GL_n(\Fpbar)\times\GL_1(\Fpbar))=G_L$;
    \item   $\rhobar(G_{L(\zeta_p)})=\GL_n(\F_{p^m})$ for some
      sufficiently large $m$ as in Lemma~{\em \ref{lem: adequacy for p^n
        sufficiently large} (}so in particular,
      $\rhobar(G_{L(\zeta_p)})$ is adequate{\em )};
  \item $\overline{L}^{\ker\rhobar}$ does not contain $L(\zeta_p)$;
  \item if $v\nmid p$ is a finite place of $L^+$, then
    $\rhobar|_{G_{L^+_v}}$ is unramified.
  \end{itemize}
\end{aprop}
\begin{proof}
  Choose $m$ such that
  $\rbar(G_K)\subset\GL_n(\F_{p^m})$ and $m$ is sufficiently large as
  in Lemma~\ref{lem: adequacy for p^n sufficiently large}, and set $\F=\F_{p^m}$. We now apply Proposition 3.2 of
  \cite{frankII}; in the notation of that result, we take
  $G=\cG_n(\F)$, and we let $E$ be any totally real field with the
  property that if $v|p$ is a place of $E$, then $E_v\cong K$. We let
  $w\nmid p$ be a finite place of $E$, and we choose a character
  $\chibar_w:G_{E_w}\to\F^\times$ such that $\chibar_w|_{I_{F_w}}$
  surjects onto $\F^\times$. We take $F=E(\zeta_p)$, and we take $S$ to be
  the set of places of $E$ which either divide $p$ or else are infinite, together
  with the place $w$. We take
  $c_v=j\in\cG_n(\F)$ for each infinite place $v$ of $S$, and for each
  place $v|p$ of $E$ we let $H_v/E_v$ correspond to
  $\overline{K}^{\ker\rbar}/K$, and let
  $\phi_v:\Gal(H_v/E_v)\to\cG_n(\F)$ correspond to
  $(\rbar,\varepsilonbar^{1-n}):G_K\to\GL_n(\F)\times\GL_1(\F)\subset\cG_n(\F)$. We
  let $H_w=\overline{E}_w^{\ker\chibar_w}$, and let
  $\phi_w=(1,\chibar_w):\Gal(H_w/E_w)\to \GL_n(\F)\times\GL_1(\F)\subset\cG_n(\F)$.

Writing $M^+/E$ for the field denoted $K$ in \cite{frankII}, we see that
$M^+$ is a totally real field, and that we have a surjective
representation $\rhobar:G_{M^+}\to\cG_n(\F)$, with the properties that
\begin{itemize}
 \item for each place $v|p$ of $M^+$, $M^+_v\cong K$ and there is a
    place $\tv$ of $L$ lying over $v$ such that $\rhobar|_{M_\tv}\cong\rbar$,
  \item $w$ splits completely in $M^+$, and if $w'|w$ then   $\nu\circ\rhobar|_{G_{M^+_{w'}}}=\chibar$,
   \item for each place $v|\infty$ of $M^+$, $\rhobar(c_v)=j$, where
     $c_v$ is a complex conjugation at $v$,
   \item $\nu\circ\rhobar|_{G_{M^+_v}}=\varepsilonbar^{1-n}$ if $v|p$,
  \item $\overline{M^+}^{\ker\rhobar}$ does not contain $M^+(\zeta_p)$.
\end{itemize}
Define $M/M^+$ by $G_M=\rhobar^{-1}(\GL_n(\F)\times\GL_1(\F))$. Since
$\rhobar(c_v)=j$ for each complex conjugation $c_v$, we see that $M$
is an imaginary CM field. Similarly, we see that each place of $M^+$
lying over $p$ or $w$ splits in $M$.

Now consider the character
$\phibar:=(\nu\circ\rhobar)\varepsilonbar^{n-1}\delta_{M/M^+}^n$. By
construction, we see that $\phibar(c_v)=1$ for each complex
conjugation $c_v$, that $\phibar|_{G_{M^+_v}}=1$ for each place $v|p$,
and that $\phibar|_{I_{M^+_{w'}}}=\chibar|_{I_{M^+_{w'}}}$ for each
place $w'|w$. Replace $M^+$ by $N^+:=\overline{M^+}^{\ker\phibar}$,
$M$ by $N:=N^+M$, and $\rhobar$ by $\rhobar|_{G_{N^+}}$. Then
$N^+/M^+$ is a totally real abelian extension in which all places
dividing $p$ split completely, and it is linearly disjoint from
$M^+(\zeta_p)$ over $M^+$ (by considering the ramification at places
above $w$). 

To complete the proof, it suffices to choose a totally real finite
Galois extension $L^+/N^+$ such that
\begin{itemize}
\item each place of $N^+$ lying over $p$ splits completely in $L^+$,
\item $\rhobar|_{G_{L^+_v}}$ is unramified for all finite places
  $v\nmid p$ of $L^+$, and
\item $L^+$ is linearly disjoint from
  $\overline{N^+}^{\ker\rhobar|_{G_{N^+}}}(\zeta_p)$ over $N^+$.
\end{itemize}
The first two of these conditions concern the splitting behaviour of a
finite number of places of $N^+$, so the existence of an extension
satisfying all but the third condition is guaranteed by Lemma 2.2 of
\cite{MR1981033}. In order to satisfy the third condition as well, for
each of the (finitely many) Galois intermediate fields
$\overline{N^+}^{\ker\rhobar|_{G_{N^+}}}(\zeta_p)/L_i/N^+$ with the
property that $\Gal(L_i/N^+)$ is simple one chooses a finite place of
$N^+$ which does not split completely in $L_i$, and arranges that it
splits completely in $L^+$ (cf.\ the proof of Proposition 6.2 of
\cite{blght}).
\end{proof}
In order to study the Breuil--M\'ezard conjecture globally, we need to
be able to realise local mod $p$ representations as part of global
automorphic Galois representations. In order to do this, we will apply
the above construction, and then show that the representation
$\rhobar$ is potentially automorphic. This essentially follows from
the results of \cite{BLGGT}. However, the main theorems of
\cite{BLGGT} make no attempt to control the local behaviour at finite places
of the extensions over which potential automorphy is proved, while it is
crucial for us that we do not change the local fields at places
dividing $p$. Fortunately, a simple trick using restriction of scalars
that we learned from Richard Taylor allows us to deduce the potential
automorphy results that we need from those of \cite{BLGGT}. We will
need the following conjecture.

\begin{aconj}\label{conj: existence of local crystalline lifts}Let
  $K/\Qp$ be a finite extension, and $\rbar:G_K\to\GL_n(\Fpbar)$ a
  continuous representation. Then $\rbar$ has a potentially
  diagonalizable lift $r:G_K\to\GL_n(\Qpbar)$ which has regular
  Hodge--Tate weights.
  \end{aconj}
  \begin{aremark}\label{rem: cases know of existence of local
      lifts}{\em If
    $n=2$ then Conjecture \ref{conj: existence of local crystalline
      lifts} is easily verified by a Galois cohomology calculation
    (for example, any potentially Barsotti--Tate lift is potentially
    diagonalizable, and the result is then immediate in the
    irreducible case, and in the reducible case is a special case of
    Lemma 6.1.6 of \cite{blggord}). We anticipate that a similar
    argument works more generally.}\end{aremark}

\begin{alemma}\label{lem: local char 0 realisation over CM, no modularity yet}
  Let $p\nmid 2n$ be prime.  Let $K/\Qp$ be a finite extension, and
  let $\rbar:G_K\to\GL_n(\Fpbar)$ be a continuous
  representation. Assume that Conjecture~{\em \ref{conj: existence of local
    crystalline lifts}} holds for $\rbar$. Let
  $\rhobar:G_{L^+}\to\cG_n(\Fpbar)$ be the representation provided by
  Lemma~{\em \ref{prop: local realisation without modularity}}. Then
  there is a lift $\rho:G_{L^+}\to\cG_n(\Qpbar)$ of $\rhobar$ such that
  \begin{itemize}
  \item
$\nu\circ \rho=\varepsilon^{1-n}\delta_{L/L^+}^n$,  \item $\rho$ is unramified
outside of places dividing $p$, and
\item if $w|p$ is a place of $L$, then $\rho|_{G_{F_w}}$ is potentially
  diagonalizable with regular Hodge--Tate weights.
  \end{itemize}
\end{alemma}
\begin{proof} This follows at once from Theorem A.4.1
  of~\cite{blggU2}, Lemma~\ref{lem: adequacy for p^n sufficiently large}
  and Conjecture~\ref{conj: existence of local crystalline lifts}.
  \end{proof}
At this point, we would like to apply Theorem 4.5.1 of \cite{BLGGT} to
establish that for some CM extension $F/L$, $\rho|_{G_F}$ is (in the
terminology of \cite{BLGGT}) automorphic. However, for our
applications we need the places of $L$ above $p$ to split completely
in $F$, which is not guaranteed by Theorem 4.5.1 of
\cite{BLGGT}. Accordingly, we need to give a slight modification of
the proof of \emph{loc.\ cit}.\ 
\begin{aprop}\label{prop: potential automorphy keeping p split}
  Maintain the notation and assumptions of Lemma~{\em \ref{lem: local char
    0 realisation over CM, no modularity yet}}. Then there is a CM
  extension $F/L$ which is linearly disjoint from
  $L^{\ker\rhobar}(\zeta_p)$ over $L$, such that each place of $L$ lying over $p$ splits
  completely in $F$, and $\rho|_{G_F}$ is automorphic in the terminology
  of \cite{BLGGT}.
\end{aprop}
\begin{proof}
  The proof is a straightforward modification of the arguments of~\cite{BLGGT},
specifically Theorem 3.1.2
  and Proposition 3.3.1 of~\emph{op.\ cit}. We sketch the details. We
  will freely use the notation and terminology of \cite{BLGGT}. 
As in the proof of Proposition 3.3.1 of \cite{BLGGT}, we may choose a
character $\psi:G_{L}\to\Qpbartimes$ such that\begin{itemize}
\item $\psi$ is crystalline at all places above $p$.
\item
  $R:=I(\rho|_{G_{L^+}}\otimes(\psi,\varepsilon^{-n}\delta_{L/L^+}^{n+1})):G_{L^+}\to\GSp_{2n}(\Qpbar)$
  has multiplier $\varepsilon^{1-2n}$, and is crystalline with distinct
  Hodge--Tate weights at all places lying over~$p$.
\item $\Rbar(G_{L^+(\zeta_p)})$ is adequate. (This follows from the
  choice of $\F=\F_{p^m}$, which was chosen large enough that the
  conclusion of Lemma~\ref{lem: adequacy for p^n sufficiently large} holds.)
\end{itemize} 
We now employ a slight variant of the proof of Theorem 3.1.2 of
\cite{BLGGT}. We let the set $\cI$ of \emph{loc.\ cit}.\ consist of
just the single element $\{1\}$, we set $n_1=2n$, and we put $l_1=p$,
$r_1:=\Rbar$, $F=F_0=L^+$, $\Favoid=L(\zeta_p)$. We then apply the
constructions of \emph{loc.\ cit}., choosing in particular an
auxiliary positive integer $N$, and constructing a geometrically
irreducible scheme $\tT=T_{\rbar_1\times\rbar'_1}$ over $\Spec
L^+(\zeta_N)^+$. 

Choose a solvable extension $(L')^+/L^+$ of
totally real fields such that
  \begin{itemize}
  \item $(L')^+$ is linearly disjoint from
    $\overline{L}^{\ker\rhobar}(\zeta_p)$ over $L^+$, and
  \item for each place $v|p$ of $(L')^+(\zeta_N)^+$, there is a point
    $P_v\in\tT((L')^+(\zeta_N)^+_v)$ with $v(t_1(P_v))<0$.
  \end{itemize}

  We now construct a geometrically irreducible scheme $\tT'$ over
  $\Spec (L')^+(\zeta_N)^+$ in exactly the same way as $\tT$ is
  constructed over $\Spec L^+(\zeta_N)^+$, and we then set
  $T':=\Res_{(L')^+(\zeta_N)^+/L^+}\tT'$, a geometrically irreducible
  scheme over $\Spec L^+$.

We can then apply
  Proposition 3.1.1 of \cite{BLGGT}, to find a finite Galois extension
  $F^+/L^+$ of totally real fields in which the places of $L^+$ over $p$
  split completely, and a point of $P\in T'(F^+)$ which satisfies the  assumptions
  (relating to linear disjointness from a certain field, and on the
   $v(T_i(P))$) on the
  field $F'$ in the proof of Theorem 3.1.2 of \cite{BLGGT}. (We can
  assume that the places of $L^+$ over $p$ split completely in $F^+$
  by the assumption on the places of $(L')^+(\zeta_N)^+$ lying over
  $p$.)

Regarding $P$ as an $F^+(L')^+(\zeta_N)^+$-point of $\tT'$, it then
  follows exactly as in \emph{loc.\ cit}.\  that there is a continuous lift
  $R':G_{F^+(L')^+(\zeta_N)^+}\to\GSp_{2n}(\Qpbar)$ of the restriction
$\rhobar|_{G_{F^+(L')^+(\zeta_N)^+}}$ such that
  \begin{itemize}
  \item $R'$ is ordinary, and
  \item $R'$ is automorphic.
  \end{itemize}
Let $M$ be a quadratic totally imaginary extension of $F^+(L')^+(\zeta_N)^+$ such
that
\begin{itemize}
\item all places of $F^+(L')^+(\zeta_N)^+$ above $p$ split in $M$, and
\item $M$ is linearly disjoint from
  $\overline{L}^{\ker\rhobar}F^+(L')^+(\zeta_p)$ over  $F^+(L')^+$.
\end{itemize}
It follows from Theorem 4.2.1 of \cite{BLGGT} that $R|_{G_{M}}$
is automorphic, and then from Lemma 2.1.1 of \cite{BLGGT} that
$\rho|_{G_{F^+(L')^+(\zeta_N)^+L}}$ is automorphic. Since the extension $F^+(L')^+(\zeta_N)^+L/F^+L$ is
solvable, it follows from Lemma 1.4 of \cite{blght} that
$\rho|_{G_{F^+L}}$ is automorphic. Since the places of $L^+$ over $p$
  split completely in $F^+$, the claim follows upon taking $F=F^+L$.
\end{proof}
\begin{acor}
  \label{cor: the final local-to-global result} Suppose that $p\nmid 2n$, that $K/\Qp$ is a finite extension, and let
  $\rbar:G_K\to\GL_n(\Fpbar)$ be a continuous representation for which
  Conjecture {\em \ref{conj: existence of local crystalline lifts}} holds. Then
  there is an imaginary CM field $F$ and a continuous irreducible
  representation $\rhobar:G_{F^+}\to\cG_n(\Fpbar)$ such that
  \begin{itemize}
  \item each place $v|p$ of $F^+$ splits in $F$, and has $F^+_v\cong K$,
  \item for each place $v|p$ of $F^+$, there is a place $\tv$ of $F$
    lying over $v$ with $\rhobar|_{G_{F_\tv}}$ isomorphic to $\rbar$,
  \item $\rhobar$ is unramified outside of $p$,
   \item $\rhobar^{-1}(\GL_n(\Fpbar)\times\GL_1(\Fpbar))=G_F$,
    \item $\rhobar(G_{F(\zeta_p)})$ is adequate,
  \item $\overline{F}^{\ker\rhobar}$ does not contain $F(\zeta_p)$,
  \item $\rhobar$ is automorphic in the sense of
    Definition~{\em \ref{defn: max ideal or mod p Galois rep is automorphic}},
    and in particular  $\mu\circ\rhobar=\varepsilonbar^{1-n}\delta^n_{F/F^+}$.
  \end{itemize}
\end{acor}
\begin{proof}
  This follows from Proposition \ref{prop: potential automorphy
    keeping p split} and the theory of base change between $\GL_n$ and
  unitary groups, cf. Proposition 2.2.7 of \cite{ger}.
\end{proof}
\bibliographystyle{amsalpha}
\bibliography{bibforrefinedbm}
\end{document}